\documentclass[12pt]{amsart}

 \usepackage[dvips]{graphicx}
\usepackage{amsmath, amssymb}
\usepackage{amscd}
\usepackage[all]{xy}
\usepackage{pifont}

\numberwithin{equation}{section}
\newtheorem{theorem}{Theorem}[section]  
\newtheorem{theorem?}{``Theorem''}[section]  
\newtheorem{corollary}[theorem]{Corollary}
\newtheorem{proposition}[theorem]{Proposition}
\newtheorem{lemma}[theorem]{Lemma}

\theoremstyle{definition}
\newtheorem{definition}[theorem]{Definition}

\newtheorem{question}{Question}

\theoremstyle{remark}
\newtheorem{remark}[theorem]{Remark}  

\newcommand{\R}{{\mathbb R}}
\newcommand{\C}{{\mathbb C}}

\newcommand{\N}{{\mathbb N}}
\newcommand{\Z}{{\mathbb Z}}

\renewcommand{\a}{\alpha}
\renewcommand{\b}{\beta}

\newcommand{\ord}{{\rm ord}}

\begin{document}
\title[Newton polyhedra and order of contact]
{
Newton polyhedra
and order of contact \\
on real hypersurfaces} 
\author{Joe Kamimoto}
\address{Faculty of Mathematics, Kyushu University, 
Motooka 744, Nishi-ku, Fukuoka, 819-0395, Japan} 
\email{
joe@math.kyushu-u.ac.jp}
\email{ }
\maketitle


\begin{abstract}
The purpose of this paper 
is to investigate order of
contact on real hypersurfaces in 
$\C^n$ by using Newton polyhedra
which are important notion 
in the study of singularity theory.
To be more precise, 
an equivalence condition for
the equality of regular type and singular
type is given by 
using the Newton polyhedron of a 
defining function for the respective hypersurface.
Furthermore, 
a sufficient condition for 
this condition, which is more useful, 
is also given.
This sufficient condition is satisfied by 
many earlier known cases 
(convex domains, 
pseudoconvex Reinhardt domains and
pseudoconvex domains whose  
regular types are 4, etc.). 
Under the above conditions, 
the values of the types can be directly seen 
in a simple geometrical
information from the Newton polyhedron. 
\end{abstract}


\tableofcontents


\section{Introduction}
Let $M$ be a ($C^\infty$) real hypersurface in $\C^n$ and  
let $p$ lie on $M$.  
Let $r$ be a local defining function for $M$ 
near $p$ ($\nabla r\neq 0$ when $r=0$).
D'Angelo \cite{Dan82}, \cite{Dan93} defined an important 
invariant $\Delta_1(M, p)$, 
the {\it singular type} of $M$ at $p$, by 
\begin{equation}
\Delta_1(M, p) : = \sup_{\gamma\in\Gamma} 
\frac{\ord(r\circ \gamma)}{\ord(\gamma-p)},
\label{eqn:1.1}
\end{equation}
and established its fundamental properties. 
Here $\Gamma$ denotes the set of (germs of) holomorphic mappings
$\gamma:(\C,0) \to (\C^n,p)$ with $\gamma\not\equiv p$.
For a $C^{\infty}$ mapping $h:\C\to\C$ or $\C^n$ such that $h(0)=0$, 
let $\ord(h)$ denote the order of vanishing of $h$ at $0$
(see Section~6 for its exact definition).

As is well-known, 
the invariant (\ref{eqn:1.1}) plays important roles in the study 
of the $\bar{\partial}$-Neumann problem. 
Let $M$ be the boundary of a 
smoothly bounded pseudoconvex domain in $\C^n$. 
It was shown by Catlin \cite{Cat83}, \cite{Cat87} that
the finite type condition at $p$ 
(i.e., $\Delta_1(M, p)<\infty$)
is equivalent to the condition that 
the subelliptic estimate holds near $p$.
Furthermore, since this invariant is deeply connected with
many analytical subjects in the study of 
several complex variables, 
its properties have been investigated from various points of view. 

Computing the singular type
(\ref{eqn:1.1}), or even determining
whether the respective type is finite, 
is not always simple matter. 
This difficulty is often caused by 
the definition in which {\it all} the mapping in $\Gamma$ 
must be treated. 
Thus, it is desirable 
that $\Gamma$ is replaced by an appropriate smaller class
which is easily treated. 
%
It has been known in special cases, below, 
that 
in order to decide the type, 
it suffices to check all the {\it regular} 
holomorphic mappings in $\Gamma$. 
To be more specific,  
let us explain this issue. 
The {\it regular type} of $M$ at $p$ is defined by 
\begin{equation}\label{eqn:1.2}
\Delta_1^{{\rm reg}}(M,p)
:=
\sup_{\gamma\in{\Gamma}_{{\rm reg}}} 
\{\ord(r\circ \gamma)
\},
\end{equation}
where $\Gamma_{{\rm reg}}$ 
denotes the set of holomorphic mappings
$\gamma:(\C,0) \to (\C^n,p)$ with 
${\rm ord}(\gamma)=1$ (i.e. $\nabla\gamma(0)\neq 0$).
Since $\Gamma_{{\rm reg}}\subset\Gamma$, 
the inequality 
$\Delta_1(M,p)
\geq
\Delta_1^{{\rm reg}}(M,p)$
always holds. 
We will consider the following question. 
\begin{question}
When does the following equality hold?
\begin{equation}\label{eqn:1.3}
\Delta_1(M,p)
=
\Delta_1^{{\rm reg}}(M,p).
\end{equation}
\end{question}
Until now, there have been results 
showing that the  equality (\ref{eqn:1.3}) holds
for $M$, 
where $M$ is the smooth boundary of the following 
domains:
\begin{enumerate}
\item[(A)] Convex domains 
(McNeal \cite{Mcn92});
\item[(B)] Star-shaped domains
(Boas-Straube \cite{BoS92});
\item[(C)] 
Pseudoconvex Reinhardt domains
(Fu-Isaev-Krantz \cite{FIK96});
\item[(D)] Pseudoconvex 
(or property {\bf PS}) 
domains with $\Delta_1^{{\rm reg}}(M,p)=4$
(McNeal-Mernik \cite{McM18}, D'Angelo \cite{Dan18}).
\end{enumerate}
The methods used in the above cited papers
are so different that it seems difficult
to find common essential characteristics 
in them
to establish the equality (\ref{eqn:1.3}).
In this paper, 
we investigate conditions for 
the equality (\ref{eqn:1.3})
by using the {\it Newton polyhedron} of 
a defining function for $M$, 
which is an important notion in the study of
singularity theory
(c.f. \cite{AGV85}, \cite{Oka97}), 
and interpret this phenomenon containing 
the above cases based on the geometry of 
the respective Newton polyhedron.    

Let us define the Newton polyhedron of 
a smooth function $F$ defined near the origin 
in $\C^n$.  
The Taylor series expansion of $F$ at the origin is 
\begin{equation}\label{eqn:1.4}
F(z,\bar{z}) \sim \sum_{\a,\, \b \in \N_0^n} 
C_{\a \b} z^\a \bar{z}^{\b} \quad \mbox{ with 
$C_{\a \b} = \dfrac{1}{\a ! \b !} 
D^{\alpha}\bar{D}^{\beta}F(0,0)$.}
\end{equation}
The {\it support} of $F$ is defined by
$S(F)=\{\a+\b\in\N_0^n:C_{\a \b}\neq 0\}.$
The {\it Newton polyhedron} of $F$ is defined by 
\begin{equation*}
{\mathcal N}_+(F) = 
\mbox{The convex hull of }
\left(\bigcup_{\a+\b\in S(F)}(\a+\b + \R_{\geq}^n)\right).
\end{equation*}
The {\it Newton diagram} ${\mathcal N}(F)$ of $F$ is defined to be 
the union of the bounded faces of ${\mathcal N}_+(F)$.
We use coordinates $(\xi)=(\xi_1,\ldots,\xi_n)$ for points in the plane
containing the Newton polyhedron.  
The following classes of functions $F$ 
simply characterized by using their Newton polyhedra 
often appear in this paper:
\begin{itemize}
\item
$F$ is called to be {\it flat}
if ${\mathcal N}_+(F)$ is an empty set.
\item
$F$ is called to be {\it convenient} 
if ${\mathcal N}_+(F)$ meets every coordinate axis.
\end{itemize}

Let $(z)=(z_1,\ldots,z_n)$ be a holomorphic coordinate
around $p$ such that $p=0$. 
Let $r$ be a local defining function for $M$ near $p$
on the coordinate $(z)$. 
For a given tuple $(M,p;(z))$, we define
a quantity $\rho_1(M,p;(z))\in\N\cup\{\infty\}$ as follows. 
If $r$ is convenient, then let 
\begin{equation}\label{eqn:1.5}
\rho_1(M,p;(z)):=\max\{\rho_j(r):j=1,\ldots,n\},
\end{equation}
where  $\rho_j(r)$ be the coordinate of 
the point at which the Newton diagram ${\mathcal N}(r)$ 
intersects the $\xi_j$-axis (see Section~2.3).
Otherwise, let $\rho_1(M,p;(z)):=\infty$.
We remark that $\rho_1(M,p;(z))$ depends on 
the chosen coordinate $(z)$, but it is independent of 
the choice of defining functions after fixing a coordinate 
(see Section~2.4).


Now, let us give an equivalence condition for 
the equality (\ref{eqn:1.3})
by using the quantity $\rho_1(M,p;(z))$. 
\begin{theorem}
\label{thm:1.1}
When $\Delta_1^{\rm reg}(M,p)<\infty$, 
the following two conditions are equivalent. 
\begin{enumerate}
\item $\Delta_1(M, p) = \Delta_1^{{\rm reg}}(M, p)$;
\item There exists a holomorphic coordinate $(z)$ at $p$ such that
$\Delta_1(M, p) =\rho_1(M,p;(z)).$
\end{enumerate}
\end{theorem}
%
Of course, when $\Delta_1^{\rm reg}(M,p)=\infty$, 
the equality (\ref{eqn:1.3}) holds.
The proof of Theorem~1.1 will be given in Section~5. 

It seems to be not so easy to see what kinds of coordinates 
satisfy the condition in (ii) in Theorem~1.1. 
In order to consider a useful class of coordinates,  
let us introduce a concept ``nondegeneracy condition''
on a smooth function $F$ defined near the origin 
in $\C^n$. 

Let $\kappa$ be a bounded face of ${\mathcal N}_+(F)$.
The $\kappa$-{\it part of} $F$ is the polynomial defined by 
\begin{equation}\label{eqn:1.6}
F_{\kappa}(z,\bar{z})=
\sum_{\a+\b \in \kappa} 
C_{\a \b} z^\a \bar{z}^{\b}.
\end{equation}
The set of complex curves $\widetilde{\Gamma}^*_{\kappa}$ 
is defined by
\begin{equation}\label{eqn:1.7}
\begin{split}
\widetilde{\Gamma}^*_{\kappa}:=
\{(c_1 t^{a_1},\ldots,c_n t^{a_n}): 
c\in(\C^*)^n, t\in\C, 
\mbox{ $a\in\N^n$ determines $\kappa$ }\},
\end{split}
\end{equation}
where
$c=(c_1,\ldots,c_n)\in(\C^*)^n$, 
$a=(a_1,\ldots,a_n)\in\N^n$ and 
``$a\in\N^n$ determines $\kappa$'' means that 
the equality $\{\xi\in{\mathcal N}_+(F):\langle a,\xi\rangle =l\}=\kappa$ 
holds for some $l\in\N$ (see Section~2.2).
\begin{definition}
Let $\kappa$ be a bounded face of ${\mathcal N}_+(F)$.
The $\kappa$-part $F_{\kappa}$ of $F$ 
is said to be {\it nondegenerate} if 
\begin{equation}\label{eqn:1.8}
F_{\kappa}\circ \gamma \not\equiv 0 
\quad \mbox{ for any $\gamma\in \widetilde{\Gamma}^*_{\kappa}$}.
\end{equation}
A function $F$ is said to be {\it nondegenerate}, 
if $F_{\kappa}$ is nondegenerate for every bounded face 
$\kappa$ of ${\mathcal N}_+(F)$. 
\end{definition}
Detailed properties of the above condition 
will be explained in Section~3. 
This condition is analogous one introduced 
by Kouchnirenko \cite{Kou76}
(see Section~3.1 (5)). 
His nondegeneracy condition often plays 
important roles in resolution of singularities
by using the geometry of Newton polyhedra. 
The nondegeneracy condition 
in Definition~1.2 is also essentially important 
in our analysis. 
Indeed, the following theorem shows that 
this condition gives sufficiently appropriate properties
for the condition (ii) in Theorem~1.1. 
\begin{theorem}
\label{thm:1.3}
If there exists a holomorphic coordinate $(z)$ at $p$ 
such that $p=0$ on which a local defining function $r$ for $M$ 
is {\it nondegenerate},
then the following equalities holds: 
\begin{equation}\label{eqn:1.9}
\Delta_1(M, p) = \Delta_1^{{\rm reg}}(M, p)=\rho_1(M,p;(z)).
\end{equation}
\end{theorem}
The above coordinate $(z)$ is said to be {\it canonical} for $M$ at $p$.  
The proof of Theorem~1.3 will be given in Section~9 below. 

From Theorem~1.3,
when the Newton polyhedron ${\mathcal N}_+(r)$ can be 
explicitly described on a canonical coordinate, 
the exact values of the regular and singular types  
(including the case of $\infty$) can be directly seen, 
which are equal to $\rho_1(M,p;(z))$.
In particular, the {\it convenience} condition 
(i.e. ${\mathcal N}_+(r)$ 
meets every coordinate axis)
determines 
whether the type of $p$ is finite.

Although many important properties of
the regular and singular types can be understood 
on canonical coordinates, 
it is another serious issue to 
determine whether
canonical coordinates exist 
for a given real hypersurface and, if they exist, 
to actually construct these coordinates. 
We consider this issue not only 
in the above mentioned cases (A), (C), (D) but also 
in more general cases.  
\begin{theorem}\label{thm:1.4}
If $M$ and $p\in M$ satisfy one of the following conditions: 
\begin{enumerate}
\item $M$ is of semiregular type (h-extendible) at $p$ (see Lemma~12.2);
\item $M$ is the boundary of a 
pseudoconvex Reinhardt domain ($p$ is any point);
\item $M$ satisfies property {\bf PS} at $p$ 
with $\Delta_1^{{\rm reg}}(M,p)=4$,
\end{enumerate}
then $M$ admits canonical coordinates at $p$.
\end{theorem}
The cases (i), (ii), (iii) in Theorem~1.4 
are respectively treated 
in Theorems 12.3, 13.4, 14.1, below,  
whose proofs will be given after their statements. 
More precise investigation is seen in Sections~12--14. 
It follows from Theorem~1.4 that Theorem~1.3 
includes the above mentioned results on (A), (C), (D). 

\begin{remark}
(1)\quad 
There exist real hypersurfaces not admitting canonical coordinates. 
For example, it is easy to see that
there is no canonical coordinates 
near the origin for the real hypersurfaces $M_1$, $M_2$ 
in $\C^3$, which are respectively defined by
\begin{equation*}
\begin{split}
&r_1(z,\bar{z})=
2{\rm Re}(z_3)+|z_1^3-z_2^2|^2, \\
&r_2(z,\bar{z})=
2{\rm Re}(z_3)+|z_1|^2|z_2|^2|z_1-z_2|^2+|z_1|^{10}+|z_2|^{10}.\\
\end{split}
\end{equation*}
The functions $r_1$, $r_2$ are plurisubharmonic
functions. 
Note that $r_1$ often appears in the studies of 
D'Angelo (\cite{Dan82}, \cite{Dan93}, etc.) 
and that 
$M_2$ belongs to the star-shaped case (B) 
(\cite{BoS92}, see Section~15).
It is easy to check the following:
\begin{eqnarray}
&&\Delta_1(M_1, 0) = \infty, \quad \Delta_1^{{\rm reg}}(M_1,0)=6; 
\label{eqn:1.10}\\ 
&&\Delta_1(M_2, 0) = 10, \quad \Delta_1^{{\rm reg}}(M_2,0)=10. 
\label{eqn:1.11}
\end{eqnarray}
It follows from (\ref{eqn:1.11}) that 
the existence of canonical coordinates is not necessary
for the equality of regular and singular types.

(2)\quad 
Owing to Theorem~1.3, 
we can easily produce
many examples of pseudoconvex 
hypersurfaces satisfying the equality (\ref{eqn:1.3}), 
which are not contained in the cases in Theorem~1.4. 
For example, it will be easily recognized 
after understanding
the investigation in Sections 12--14 that 
the pseudoconvex hypersurface defined by
\begin{equation*}
{\rm Re}(w)+|z_1|^8+
\frac{15}{7}|z_1|^2({\rm Re}(z_1))^6+|z_1 z_2|^2
+|z_2|^6=0
\end{equation*}
is not contained in any case in Theorem~1.4, 
but it admits a canonical coordinate at the origin 
and its regular and singular types of the origin are $8$.
Moreover, 
this example does not satisfy the hypothesis 
in the result (see Theorem~15.1) in \cite{BoS92}.
\end{remark}

Let us explain ideas of our analysis roughly. 
The substantial analysis of the types is to 
investigate the following order.
Let $F$ be a smooth function 
defined near the origin in $\C^n$ with $F(0)=0$. 
The {\it order of contact of $\gamma$ with} $F$ is defined by 
\begin{equation}\label{eqn:1.12}
O(F,\gamma):=\frac{\ord(F\circ \gamma)}{\ord(\gamma)} 
\mbox{\,\,\,\,\, for $\gamma\in\Gamma$,}
\end{equation}
where $\Gamma$ is the same as it in (\ref{eqn:1.1}) with $p=0$.
In order to understand the situation
of the contact of $\gamma$ with $F$ in (\ref{eqn:1.12}), 
we use the geometry of the Newton polyhedron of $F$ 
and the vector $\phi(\gamma)$ naturally defined 
by the order of vanishing of $\gamma$.
Roughly speaking, 
the ``distance'' of the Newton polyhedron of $F$ from the origin
expresses the ``flatness'' at the origin
of the hypersurface $M$ defined by $F=0$. 
When the Newton polyhedron of $F$ becomes far from the origin, 
the flatness of $M$ becomes strong. 
Our issue about (\ref{eqn:1.12}) is more complicated. 
The situation of 
flatness of the restriction of $M$ to the complex curve 
defined by $\gamma$ is investigated. 
We introduce a new quantity ``distance
of the Newton polyhedron of $F$ in the direction $\phi(\gamma)$'' 
and show that the order of contact (\ref{eqn:1.12}) equals 
this distance 
under the nondegeneracy condition on $F$ in Definition~1.2.
Furthermore, this distance can be clearly expressed
by more simple geometrical information from
the Newton polyhedron of $F$.
Applying this expression to the computation of the two types
in (\ref{eqn:1.1}), (\ref{eqn:1.2}),
we can see that, 
under the nondegeneracy condition of a defining function, 
the regular and singular types agree and 
that they can be expressed 
by using a simple geometrical information from 
the Newton polyhedron.

The technique of using Newton polyhedra has 
many significant applications in singularity theory
(c.f. \cite{AGV85}, \cite{Oka97}). 
In particular, this technique has been great success 
in the study of the \L ojasiewicz exponent
(\cite{Lic81}, \cite{Fuk91}, \cite{Ole13}, \cite{Oka18}, etc.).
It is known in \cite{LeT08} 
that the \L ojasiewicz exponent of $F$ can be expressed 
in a similar form to (\ref{eqn:1.1}). 
Roughly speaking, this exponent 
can be written as the supremum of the form (\ref{eqn:1.12}) 
by replaced $F$ by $\nabla F$
where $F$ is holomorphic. 
Our study about the types is   
analogous to the above cited works
on the \L ojasiewicz exponent.  
There are also many interesting applications of 
Newton polyhedra to the other analytical subjects. 
We only refer for studies 
about the oscillatory integrals to 
\cite{Var76}, \cite{PhS97}, \cite{IkM16}, \cite{KaN16}, 
etc.
and for those about 
the Bergman kernel to 
\cite{Kam04}, \cite{CKO04}, \cite{ChF12}, etc.
As for study about order of contact 
approached from the singularity theory, 
there are important works due to  
McNeal-N\'emethi \cite{McN05}, G. Heier \cite{Hei08} and 
Forn\ae ss and Stens\o nes \cite{FoS10}.
This paper mainly treats Question 1 but we 
believe that 
the Newton polyhedron technique may be applied to 
the other deep problems around the types.  

This paper is organized as follows.
In Sections~2--3, 
detailed properties and subtle remarks about 
the Newton polyhedron and 
the nondegeneracy condition 
are explained. 
In Section~4, 
we introduce 
a new invariant from the quantity (\ref{eqn:1.5}) and 
show that this invariant equals the regular type in
(\ref{eqn:1.2}).
Theorem~1.1 is proved in Section~5. 
Sections 6--8 are the most important parts 
in our analysis, which show that 
the order $O(F,\gamma)$ in (\ref{eqn:1.4}) 
can be expressed by using the geometry of 
the Newton polyhedron of $F$
under the nondegeneracy condition. 
By using results obtained in Section~8,  
Theorem~1.3 can be easily shown in Section~9.
In Sections~10--14, 
special cases satisfying the hypothesis
in Theorem~1.3
are precisely investigated and, 
as a result, Theorem~1.4 is proved.  
In Section~15, 
the result on the star-shaped case in (B) 
due to Boas and Straube \cite{BoS92},
which generalizes the result \cite{Mcn92} 
and is not included in Theorem~1.3, 
is discussed. 

In \cite{Kam19}, part of results of this paper
has been announced. 
\vspace{1 em}

{\it Notation, symbols and terminology.}\quad
\begin{itemize}
\item 
The following symbols are used: 
\begin{equation*}
\begin{split}
&\N_{0}:=\{0\}\cup\N, \,\, \R_{>}:=\{x\in\R:x > 0\}, \,\, \\
&\R_{\geq}:=\{x\in\R:x \geq 0\}, \,\,
\C^*:=\C\setminus\{0\}.
\end{split}
\end{equation*} 
\item
Some specific vectors are denoted as follows.
\begin{equation}\label{eqn:1.13}
\begin{split}
&0:=(0,\ldots,0),\,\,
{\bf 1}:=(1,\ldots,1), \,\,
{\mbox{\boldmath $\infty$}}:=(\infty,\ldots,\infty),
\\ 
&{\bf e}_j:=(0,\ldots,\stackrel{(j)}{1},\ldots,0) \,\,\,
\mbox{ for \, $j=1,\ldots,n$.}
\end{split}
\end{equation}
\item
The multi-indices are used as follows.
For $z=(z_1,\ldots,z_n), \,\,
\bar{z}=(\bar{z}_1,\ldots,\bar{z}_n), 
\in\C^n$, 
$\a=(\a_1,\ldots,\a_n), \b=(\b_1,\ldots,\b_n)\in\N_0^n$, 
define
\begin{eqnarray*}
&& 
z^{\a}:=z_1^{\a_1}\cdots z_n^{\a_n}, \,\,
\bar{z}^{\b}:=\bar{z}_1^{\b_1}\ldots\bar{z}_n^{\b_n}, \\
&&
|\a|:=\a_1+\cdots+\a_n, \quad
\a!:=\a_1 !\cdots \a_n!, \quad
0 !:=1, \\ 
&&
D^{\alpha}:=
\frac{\partial^{|\alpha|}}
{\partial z_1^{\alpha_1}\cdots \partial z_n^{\alpha_n}}, \quad
\bar{D}^{\beta}:=
\frac{{\partial}^{|\beta|}}
{\partial \bar{z}_1^{\beta_1}\cdots \partial \bar{z}_n^{\beta_n}}.
\end{eqnarray*}
\item
For $\xi=(\xi_1,\ldots,\xi_n)$, 
$\zeta=(\zeta_1,\ldots,\zeta_n)\in\R^n$, 
we denote 
$
\langle \xi,\zeta \rangle
=\sum_{j=1}^n \xi_j \zeta_j.
$
\item
In this paper, 
we always consider smooth functions, 
mappings, real hypersurfaces and complex curves 
as their respective germs without any mentioning. 
The following rings of germs of $\C$-valued functions 
are considered: 
\begin{itemize}
\item 
$C^{\infty}_0(\C^n)$ 
is the ring of germs of $C^{\infty}$ functions 
at the origin in $\C^n$. 
\item 
${\mathcal O}_0(\C)$ is the ring of germs of 
holomorphic functions at the origin in $\C$.
\end{itemize}
\item
We use coordinates $(\xi)=(\xi_1,\ldots,\xi_n)$ for points in the plane
containing the Newton polyhedron 
in order to distinguish this plane from the $z$-plane. 
\item
Let $F\in C^{\infty}_0(\C^n)$ and 
its Taylor series is as in (\ref{eqn:1.4}). 
We write $j_N F$ for the $N$-th order Taylor polynomial of $F$ 
at the origin, i.e., 
\begin{equation}\label{eqn:1.14}
j_N F(z,\bar{z})=
\sum_{|\a+\b| \leq N} 
C_{\a \b} z^\a \bar{z}^{\b}.
\end{equation}
When 
$\{\xi\in\R_{\geq}^n:|\xi|\leq N\}
\cap{\mathcal N}_+(F)=\emptyset$, 
set $j_N F \equiv 0$.
The {\it principal part} $F_0$ of $F$ is the polynomial 
defined by 
\begin{equation}\label{eqn:1.15}
F_{0}(z,\bar{z})=
\sum_{\a+\b \in {\mathcal N}(F)} 
C_{\a \b} z^\a \bar{z}^{\b}.
\end{equation}
\item
We use the words {\it pure terms} for any harmonic
polynomial and {\it mixed terms}
for any sum of monomials that are
neither holomorphic nor anti-holomorphic. 
\end{itemize}


\section{Elementary properties of Newton polyhedra}

\subsection{Polyhedra}
In order to treat delicate properties of 
Newton polyhedra, 
it is necessary to use many kinds of terminology 
concerning convex geometry.
Refer to \cite{Zie95}   
for general theory of convex polyhedra.  

For $(a,l)\in \Z^n\times\Z$, 
let $H(a,l)$ and $H_{+}(a,l)$ be  
a hyperplane and 
a closed halfspace in $\R^n$ 
defined by
\begin{equation}\label{eqn:2.1}
\begin{split}
&H(a,l):=\{\xi\in\R^n:\langle a,\xi\rangle =l\},\\
&H_+(a,l):=\{\xi \in\R^n:\langle a,\xi \rangle \geq l\},
\end{split}
\end{equation} 
respectively. 
A ({\it convex}) {\it polyhedron} is  
an intersection of closed halfspaces:
a set $P\subset\R^n$ presented in the form
\begin{equation*}
P=\bigcap_{j=1}^N H_+(a^j,l_j)
\end{equation*}
for some $a^1,\ldots,a^N \in \Z^n$ and 
$l_1,\ldots,l_N \in \Z$.
A polyhedron $P$ is {\it bounded} if 
$P$ does not contain 
a ray $\{x+\lambda y:\lambda>0\}$
for any $x\in P$ and $y\neq 0$.

Let $P$ be a polyhedron in $\R^n$. 
A pair $(a,l)\in \Z^n\times\Z$ is said to be 
{\it valid} for $P$ 
if $P$ is contained in $H_+(a,l)$.
A {\it face} of $P$ is any set of the form 
\begin{equation}\label{eqn:2.2}
\kappa=P\cap H(a,l),
\end{equation}
where $(a,l)$ is valid for $P$. 
Since $(0,0)$ is always valid, 
we consider $P$ itself as a trivial face of $P$;
the other faces are called {\it proper faces}. 
Considering the valid pair $(0,-1)$, 
we see that the empty set is always a face of $P$. 
Indeed, $H_+(0,-1)=\R^n$, but $H(0,-1)=\emptyset$.
For a nonempty proper face $\kappa$ and 
a nonzero vector $a\in\Z^n$, 
we say that $a$ {\it determines} $\kappa$ 
if $(a,l)$ is valid for $P$ and 
the equation (\ref{eqn:2.2}) holds for some 
$l\in\Z$. 
Without any mentioning, 
the above vector $a=(a_1,\ldots,a_n)$ is 
always chosen to be primitive, 
i.e., 
${\rm gcd}(a_1,\ldots,a_n)=\pm 1$.
On the other hand, 
it is easy to see that any face is a polyhedron. 
The {\it dimension} of a face $\kappa$ is the dimension of 
its affine hull
(i.e., the intersection of all affine flats that 
contain $\kappa$). 
The faces of dimensions $0$ and $\dim(P)-1$
are called {\it vertices} and {\it facets}, respectively.
The vertices are denoted by boldfaces  
(${\bf v}_j$, ${\bf e}_j$, etc.).

For a subset $A$ of $\R^n$, the {\it convex hull} of $A$ is
the intersection of all convex sets that contain $A$: 
\begin{equation*}
{\rm conv}(A):=
\bigcap\{K\subset\R^n:
A\subset K, \,
\mbox{ $K$ convex}
\}.
\end{equation*}
A proper face $\kappa$ of $P$ 
can be expressed as 
$\kappa={\rm conv}(\{{\bf v}_1,\ldots,{\bf v}_m\})$
where 
${\bf v}_1,\ldots,{\bf v}_m$
are the vertices of $\kappa$.


\subsection{Newton polyhedra}

Let $F\in C^{\infty}_0(\C^n)$. 
The definition of the Newton polyhedron
of $F$ has been given in the Introduction. 
Let us explain many useful properties
of Newton polyhedra. 
The case when $F(0,0)\neq 0$ is trivial
(when ${\mathcal N}_+(F)=\R_{\geq}^n$ and ${\mathcal N}(F)=\{0\}$). 
Hereafter in this paper, we always assume that $F(0,0)=0$.  

Let $a\in\N_0^n$ and let $\kappa$ be a proper face of 
${\mathcal N}_+(F)$. 
When ``$a$ determines $\kappa$'', 
we give some remarks 
about the relationship between $a$ and $\kappa$. 
%
\begin{enumerate}
\item For a given nonzero vector $a\in\N_0^n$, 
a proper face $\kappa$ of ${\mathcal N}_+(F)$
can be uniquely determined 
through the equation (\ref{eqn:2.2}).
\item Conversely, 
it is easy to see that 
for a given proper face $\kappa$ of ${\mathcal N}_+(F)$, 
there may be many vectors $a\in\N_0^n$
determining $\kappa$ through (\ref{eqn:2.2}).
Note that 
$\kappa$ is a facet of ${\mathcal N}_+(F)$ if and only if
the vector $a$ is uniquely determined by $\kappa$.
\item
It will be shown in Lemma~2.1 that
$\kappa$ is a bounded face if and only if 
$a$ belongs to $\N^n$ 
(i.e., each component of $a$ is positive).   
\end{enumerate}

It is known in \cite{Zie95} that  
Newton polyhedra are polyhedra, i.e.,
Newton polyhedra can be expressed as an intersection of 
finitely many closed half spaces.  
To be more precise, 
${\mathcal N}_+(F)$ has the following properties. 
\begin{lemma}
There exist finitely many valid pairs $(a_j,l_j)\in\N_{0}^n\times\N$ 
for $j=1,\ldots, m$ such that
$${\mathcal N}_+(F)=\bigcap_{j=1}^m H_+(a_j,l_j).$$ 
Furthermore, 
for a proper face $\kappa$ of the polyhedron ${\mathcal N}_+(F)$ 
defined by a valid pair $(a,l)$, 
the following three conditions are equivalent:
\begin{enumerate}
\item The face $\kappa$ is bounded; 
\item There are finitely many points in $\kappa\cap\N_0^n$;
\item Every component of $a$ is positive (i.e., $a\in\N^n$).
\end{enumerate}
\end{lemma}
\begin{proof}
It is easy to see that 
$(a_j,l_j)$ belong to $\N_{0}^n\times\N$ for $j=1,\ldots, m$.

Let us show the equivalence of the three conditions. 
Since the equivalence of (i) and (ii) is obvious, 
we will show the implications 
(iii) $\Rightarrow$ (i) and
(ii) $\Rightarrow$ (iii).

((iii) $\Longrightarrow$ (i).)\quad 
Let $\kappa$ be decided by the valid pair
$(a,l)\in\N^n\times\N$.
Since $a_j>0$ for all $j$, 
the set $H(a,l)\cap\R_{\geq}^n$ is 
bounded polyhedron containing $\kappa$, 
which implies the boundedness of $\kappa$.

((ii) $\Longrightarrow$ (iii).)\quad
Let us assume that the condition (iii) does not hold.
Without loss of generality, 
we may set $a_1=0$. Since 
$\kappa=P\cap\{\xi:\sum_{j=2}^n a_j\xi_j=l\}$,
if $\xi\in\N_0^n$ is contained in $\kappa$,
then $\xi+\{(j,0,\ldots,0)\}$ is also contained 
in $\kappa$ for all $j\in\N$, 
which is a contradiction to (ii).  
\end{proof}

For a bounded face $\kappa$ of ${\mathcal N}_+(F)$, 
the $\kappa$-part $F_{\kappa}$ of $F$
has a good homogeneity, 
which is useful for later investigation. 
For 
$z=(z_1,\ldots,z_n)\in\C^n$, 
$\zeta\in\C$, 
$a=(a_1\ldots,a_n)\in\N_0^n$,
denote 
\begin{equation}\label{eqn:2.3}
\zeta^a\bullet z:=
(\zeta^{a_1}z_1,\ldots,\zeta^{a_n}z_n).
\end{equation}
\begin{lemma}
Let $\kappa$  be a bounded face 
of ${\mathcal N}_+(F)$ defined by 
a valid pair $(a,l)\in\N^n\times\N$.
Then, $F_{\kappa}$ has the quasihomogeneous property:
\begin{equation}\label{eqn:2.4}
F_{\kappa}(r^{a}\bullet z,\overline{r^{a}\bullet z})=
r^lF_{\kappa}(z,\bar{z}),
\quad
\mbox{for all $r\geq 0$.}
\end{equation}
Furthermore, if $F_{\kappa}$ is holomorphic, then
it has the following stronger property:
\begin{equation}\label{eqn:2.5}
F_{\kappa}(\zeta^{a}\bullet z)=
\zeta^lF_{\kappa}(z),
\quad
\mbox{for all $\zeta\in\C$.}
\end{equation}
\end{lemma}
\begin{proof}
From the definition of the $\kappa$-part (\ref{eqn:1.6}), 
we have
\begin{equation}\label{eqn:2.6}
\begin{split}
&F_{\kappa}(r^{a}\bullet z,\overline{r^{a}\bullet z})
=
\sum_{\a+\b \in \kappa} 
C_{\a \b}\cdot \prod_{j=1}^n 
(r^{a_j}z_j)^{\alpha_j} \cdot
\prod_{j=1}^n
(r^{a_j}\bar{z}_j)^{\beta_j}\\
&\quad\quad=
r^{\langle a, \a+\b \rangle}
\left(
\sum_{\a+\b \in \kappa} 
C_{\a \b} z^\a \bar{z}^{\b}
\right)=r^{\langle a, \a+\b \rangle} F_{\kappa}(z,\bar{z}). 
\end{split}
\end{equation}
Since the face $\kappa$ is determined by $H(a,l)$, 
the equality: $\langle a, \a+\b \rangle=l$ always holds if 
$\a+\b \in \kappa$.
This implies (\ref{eqn:2.4}).

When $f$ is holomorphic, $\beta_j=0$ in (\ref{eqn:2.6}), 
so the equation (\ref{eqn:2.5}) can be samely shown. 
\end{proof}

Newton polyhedra keep some properties after being multiplied by
non-zero functions.

\begin{lemma}
Let $h\in C_0^{\infty}(\C^n)$ satisfy $h(0,0)\neq 0$. 
Then we have
\begin{enumerate}
\item ${\mathcal N}_+(F)={\mathcal N}_+(hF)$;
\item For a bounded face $\kappa$ of ${\mathcal N}_+(F)$, 
the $\kappa$-part of the function $h(z,\bar{z})F(z,\bar{z})$ is
$h(0,0)F_{\kappa}(z,\bar{z})$. 
\end{enumerate}
\end{lemma}

\begin{proof}
For general $g\in C^{\infty}_0(\C^n)$, 
it is easy to see that 
${\mathcal N}_+(gF)\subset{\mathcal N}_+(F)$. 
By regarding vertices as bounded faces, 
this inclusion shows that
(ii) implies (i).
Therefore, we only show (ii).

Let $(a,l)$ be a valid pair defining $\kappa$.
Then the boundedness of $\kappa$ implies
that every component of $a$ is positive 
(see Lemma~2.1). 
The positivity of $a$ and Taylor's formula imply that
\begin{equation}\label{eqn:2.7}
\begin{split}
& h(z,\bar{z})=h(0,0)+\sum_{j=1}^n z_j h_j(z,\bar{z}) + 
\sum_{j=1}^n \bar{z}_j \overline{h_j(z,\bar{z})};\\
& F(z,\bar{z})=F_{\kappa}(z,\bar{z})+R_{\kappa}(z,\bar{z}),
\end{split}
\end{equation}
where $h_j,\bar{h}_j\in C^{\infty}_0(\C^n)$ and
$R_{\kappa}\in C^{\infty}_0(\C^n)$ satisfying that
${\mathcal N}_+(R_{\kappa})\subset H_+(a,l+1)$
with a valid pair $(a,l)$ defining $\kappa$. 
From (\ref{eqn:2.7}), 
noticing that ${\mathcal N}_+(z_j h_j F_{\kappa}),
{\mathcal N}_+(\bar{z}_j \bar{h}_j F_{\kappa})\subset 
H^{+}(a,l+1)$ from the positivity of $a$, 
we have
\begin{equation*}
{\mathcal N}_+(hF-h(0)F_{\kappa})\subset H^{+}(a,l+1),
\end{equation*}
which shows (ii). 
\end{proof}

\subsection{Some quantities related to Newton polyhedra}

Let $F\in C^{\infty}_0(\C^n)$. 
When ${\mathcal N}_+(F)$ meets the $\xi_j$-axis,
let $\rho_j(F)$ be the coordinate of the point at which
the Newton diagram ${\mathcal N}(F)$ intersect the $\xi_j$-axis,
that is, 
\begin{equation}\label{eqn:2.8}
\rho_j(F)=
\min\{\xi_j>0:(0,\ldots,0,\xi_j,0,\ldots,0)\in{\mathcal N}_+(F)\}.
\end{equation}
If ${\mathcal N}_+(F)$ does not meet the $\xi_j$-axis, 
then we set $\rho_j(F):=\infty$. 
The $n$-tuple of numbers $\rho(F)$ is defined by
\begin{equation}\label{eqn:2.9}
\rho(F)=(\rho_1(F),\ldots,\rho_n(F)) \in (\N\cup\{\infty\})^n.
\end{equation}
(When $F$ is convenient, the definition of $\rho_j$ is 
the same as it in the Introduction.)

In the special cases of $F$, the following are easy to see.
\begin{itemize}
\item If $F$ is flat, then 
$\rho(F)=(\infty,\ldots,\infty)$.
\item $F$ is convenient if and only if every $\rho_j(F)$ 
is a positive integer ($<\infty$). 
\end{itemize}

\subsection{Newton polyhedra associated to real hypersurfaces}

Let $M$ be a smooth real hypersurface in $\C^n$ 
and let $p$ lie on $M$. 
Let $(z)$ be a holomorphic coordinate around $p$ such that $p=0$ and  
let $r$ be a local defining function for $M$ near $p$
on the coordinate $(z)$. 
We respectively define the {\it Newton polyhedron} and 
the {\it Newton diagram} 
with respect to $(M,p;(z))$ by 
\begin{equation}\label{eqn:2.10}
{\mathcal N}_+(M,p;(z)):={\mathcal N}_+(r),\quad\quad
{\mathcal N}(M,p;(z)):={\mathcal N}(r).
\end{equation}
Moreover, the $n$-tuple of numbers $\rho(M,p;(z))$ 
is defined by 
\begin{equation}\label{eqn:2.11}
\rho(M,p;(z))=(\rho_1(M,p;(z)),\ldots,\rho_n(M,p;(z)))
\in (\N\cup\{\infty\})^n
\end{equation}
with $\rho_j(M,p;(z))=\rho_j(r)$ for $j=1,\ldots,n$. 
We remark that these are well-defined.
Indeed,  
since another defining function for $M$ at $p$ can be written 
as $h(z,\bar{z})r(z,\bar{z})$ with a positive 
$C^{\infty}$ function $h$, 
Lemma~2.3 (i) 
implies that the shape of the Newton polyhedron 
is independent of the choice of defining functions
after fixing the coordinates.
In particular, $\rho_1(M,p;(z))$ has the same property. 

By changing the order of variables if necessary, 
we can always choose a coordinate $(z)$ such that 
\begin{equation}\label{eqn:2.12}
\rho_1(M,p;(z))\geq \cdots \geq \rho_n(M,p;(z)).
\end{equation}
Hereafter in this paper, 
we only consider these coordinates
without any mentioning. 
It is easy to see the following:
\begin{enumerate}
\item $\rho_n(M,p;(z))=1$.
\item A defining function for $M$ is convenient 
on the coordinate $(z)$
if and only if 
$\rho_1(M,p;(z))<\infty$.
\end{enumerate}


\section{Remarks on nondegeneracy condition}

Let $F\in C^{\infty}_0(\C^n)$.
The definition of the nondegeneracy condition on $F$ 
has been given in the Introduction. 
Since this condition plays important roles 
in our analysis, 
its useful properties will be precisely explained.  
 
The following four sets of compex curves are defined, 
which will be often used in this paper.  
\begin{equation}\label{eqn:3.1}
\begin{split}
&\Gamma:=\{\gamma=(\gamma_1,\ldots,\gamma_n):
\gamma_j\in{\mathcal O}_0(\C), 
\gamma_j(0)=0 \mbox{ for $j=1,\ldots,n$}\}
\setminus\{0\}, \\
&\Gamma^*:=\{\gamma=(\gamma_1,\ldots,\gamma_n)
\in\Gamma:
\gamma_j\not\equiv 0 
\mbox{ for $j=1,\ldots,n$} \}, \\ 
&\widetilde{\Gamma}^*:=
\{(c_1 t^{a_1},\ldots,c_n t^{a_n})\in\Gamma^*: 
c\in(\C^*)^n, 
a\in\N^n, t\in\C \}, \\
&\widetilde{\Gamma}^*_{\kappa}:=
\{(c_1 t^{a_1},\ldots,c_n t^{a_n})
\in\widetilde{\Gamma}^*:
\mbox{ $a$ determines $\kappa$}\},
\end{split}
\end{equation}
where
$c=(c_1,\ldots,c_n)\in(\C^*)^n$, 
$a=(a_1,\ldots,a_n)\in\N^n$
and $\kappa$ is a bounded face of ${\mathcal N}_+(F)$ 
($\Gamma$ with $p=0$ and $\widetilde{\Gamma}^*_{\kappa}$ are
the same as those in the Introduction). 

\subsection{Remarks on Definition~1.2}
There are many remarks on the definition of 
nondegeneracy condition.

(1) \quad
By using the map $\Phi_*$ 
defined below as in (\ref{eqn:7.5}), 
the set $\widetilde{\Gamma}^*_{\kappa}$ is simply expressed as
$\widetilde{\Gamma}^*_{\kappa}=
\{\gamma\in\widetilde{\Gamma}^*:
\Phi_*(\gamma)=\kappa\}.$

(2)\quad
The set $\widetilde{\Gamma}^*_{\kappa}$ seems complicated, 
but it often suffices to check the following conditions to 
see the nondegeneracy of $F_{\kappa}$:
\begin{enumerate}
\item[(a)] 
The restriction of the zero variety: 
\begin{equation}\label{eqn:3.2}
V(F_{\kappa}):=\{z\in \C^n:F_{\kappa}(z,\bar{z})=0\}
\end{equation}  
to $(\C^*)^n\cup\{0\}$ 
does not contain any complex curves
through the origin;
\item[(b)] $F_{\kappa}\circ\gamma\not\equiv 0$ for any 
$\gamma\in\Gamma^*$;
\item[(c)] $F_{\kappa}\circ\gamma\not\equiv 0$ for any 
$\gamma\in\widetilde{\Gamma}^*$.
\end{enumerate}
By Local Parametrization theorem (\cite{Gun70}), 
(a) and (b) are equivalent. 
Moreover, the following implications hold:
(b) $\Longrightarrow$ (c) $\Longrightarrow$ (\ref{eqn:5.1}). 


(3)\quad
In order to check the nondegeneracy condition 
for a given $F$, 
one must consider not only the facets 
(i.e., $(n-1)$-dimensional faces) 
but also
{\it arbitrary dimensional} faces
of the Newton polyhedron ${\mathcal N}_+(F)$. 
For example, let us consider the case of 
the function:
$F(z,\bar{z})=
|z_1|^4-2|z_1 z_2|^2+|z_2|^4+|z_3|^4
\in C_0^{\infty}(\C^3)$.
Let 
$$\kappa_1:=
{\rm conv}(\{4{\bf e}_1,4{\bf e}_2,4{\bf e}_3\})\,\,
\mbox{ and } \,\,
\kappa_2:=
{\rm conv}(\{4{\bf e}_1,4{\bf e}_2\}).
$$
It is easy to see 
that $\kappa_1$ is the facet of ${\mathcal N}_+(F)$
and $F_{\kappa_1}$ is nondegenerate; 
while 
$\kappa_2$ is a one-dimensional face of ${\mathcal N}_+(F)$ 
and
$F_{\kappa_1}$ is degenerate. 
Thus, $F$ is degenerate. 


(4)\quad
In the definition of the nondegeneracy condition, 
$F_{\kappa}$ may vanish at some set in $(\C^{*})^n$ 
which does not have the complex structure. 
For example, let us consider the $C^{\infty}$ functions:
\begin{equation}\label{eqn:3.3}
\begin{split}
&F_1(z,\bar{z})=|z|^8+\frac{15}{7}|z|^2{\rm Re}(z^6) 
\quad \mbox{on $\C$};\\
&F_2(z,\bar{z})=
|z_1|^8+ \frac{15}{7}|z_1|^2{\rm Re}(z_1^6)+
|z_2|^8+ \frac{15}{7}|z_2|^2{\rm Re}(z_2^6) 
\quad \mbox{on $\C^2$}.
\end{split}
\end{equation} 
They vanish on some sets in $\C^*$ or $(\C^*)^2$, 
but these sets contain no complex curves. 
The function $F_1$ appears 
in a famous example of Kohn-Nirenberg in \cite{KoN73}. 
Note that the above functions are 
real-valued plurisubharmonic functions.


(5)\quad
Let us recall the definition
of the nondegeneracy condition
introduced by Kouchnirenko \cite{Kou76}, 
which plays important roles in the study of 
singularity theory (see \cite{AGV85}, \cite{Oka97}). 
Let $F$ be a holomorphic function defined 
near the origin in $\C^n$. 
The function $F$ is said to be {\it nondegenerate} 
in the sense of Kouchnirenko if 
\begin{equation}\label{eqn:3.4}
\left(
\dfrac{\partial F_{\kappa}}{\partial z_1},\cdots,
\dfrac{\partial F_{\kappa}}{\partial z_n}
\right)\neq 0 \,\,\mbox{ on } \,\,(\C^*)^n
\end{equation} 
for any bounded face $\kappa$ of ${\mathcal N}_+(F)$.
The following Euler identity equation is followed from the equation 
(\ref{eqn:2.5}): 
$\sum_{j=1}^n a_j z_j 
\dfrac{\partial F_{\kappa}}{\partial z_j}=l F_{\kappa}$ 
for 
$z\in\C^n$,
where $(a,l)\in\N^n\times\N$ is a valid pair determining the face 
$\kappa$. 
From this equation, 
if $\frac{\partial F_{\kappa}}{\partial z_j}$ 
has a common zero on $(\C^*)^n$ for $j=1,\ldots,n$, 
then $F_{\kappa}$ vanishes there.
Therefore, our nondegeneracy condition implies 
that of Kouchnirenko. 
But, it follows from Lemma~3.4, below, that   
almost all holomorphic functions are degenerate in our sense, 
so this implication is not useful. 
In other words, our nondegeneracy condition 
makes sense only in the mixed variables case. 


(6)\quad
The notion ``nondegeneracy condition'' in Definition~1.2
is not invariant under biholomorphic maps. 
For example, $|z_1|^2+|z_2|^4$ is nondegenerate, but
$|z_1-z_2|^2+|z_2|^4$ is degenerate.
Note that Kouchnirenko's nondegeneracy is also in the 
same situation. Consider the functions:
$z_1^2+z_2^4$ and $(z_1-z_2)^2+z_2^4$. 
Therefore, 
when the chosen coordinate should be strongly paid attention, 
we write 
``$F$ is nondegenerate on the coordinate $(z)$.''


\subsection{Elementary properties of nondegeneracy conditions}

The nondegeneracy property remains 
after being multiplied by nonzero functions.
\begin{lemma}
Let $h$ be a real-valued smooth function with $h(0,0)\neq 0$. 
If $F$ is nondegenerate, then so is $h(z,\bar{z})F(z,\bar{z})$. 
\end{lemma}
\begin{proof}
This is trivial from Lemma~2.3 (ii).
\end{proof}

\begin{remark}
If there exists a canonical coordinate $(z)$ for $M$ at $p$, 
then every local defining function for $M$ 
is nondegenerate on $(z)$. 
\end{remark}

The one-dimensional case is obvious.

\begin{lemma}
Every nonflat smooth function $F$ 
of one variable is nondegenerate. 
\end{lemma}
\begin{proof}
Let us assume that $F$ is degenerate. 
Then 
the principal part $F_0$ of $F$ 
(see (\ref{eqn:1.15})) 
must identically equal zero, 
which is a contradiction. 
\end{proof}

More generally, 
let us consider the case of the ${\bf v}$-part of $F$,
where ${\bf v}$ is a vertex of ${\mathcal N}_+(F)$.
In the one-dimensional case, 
$F_{\bf v}$ is always nondegenerate from the above lemma. 
But, 
$F_{\bf v}$ may be degenerate in the multi-dimensional case. 
Indeed, consider the two-dimensional example:
$
F_{\bf v}(z,\bar{z})=
z_1^2 \bar{z}_1^8 z_2^6 \bar{z}_2^4-
z_1^4 \bar{z}_1^6 z_2^4 \bar{z}_2^6.
$
Note that ${\bf v}=\{(10,10)\}$ and 
$F_{\bf v}\circ \gamma \equiv 0$, where 
$\gamma(t)=(t,t)$. 

The following lemma shows that there are few 
pluriharmonic functions satisfying the nondegeneracy condition. 

\begin{lemma}
Let $F$ be pluriharmonic near the origin. 
Then the following three conditions are equivalent.
\begin{enumerate}
\item $F$ is nondegenerate;
\item The Newton diagram ${\mathcal N}(F)$ consists of only one vertex 
in $\N_0^n$;
\item 
There exist $C_1,\ldots,C_n \in\N_0$ 
such that 
${\mathcal N}_+(F)=
\{\xi\in\R_{\geq}^n:\xi_j\geq C_j 
\mbox{ for $j=1,\ldots,n$}\}$.
\end{enumerate}
In particular, when $F$ is holomorphic or antiholomorphic, 
the same equivalences are established. 
\end{lemma}
\begin{proof}
It suffices to consider the case 
when $F$ is a holomorphic function.
Since the implications 
(ii) $\Longrightarrow$ (i) and 
(ii) $\Longleftrightarrow$ (iii) are obvious,  
we only show that the implication (i) $\Longrightarrow$ (ii). 

Let us assume that (ii) does not hold. 
Then, ${\mathcal N}_+(F)$ has 
a one-dimensional bounded face $\kappa$.   
It is easy to see that 
$F_{\kappa}$ vanishes at some point 
$c=(c_1\ldots,c_n)\in(\C^{*})^n$.
It follows from 
the equation (\ref{eqn:2.5}) in Lemma~2.2 that 
$F_{\kappa}(c_1 t^{a_1},\ldots,c_n t^{a_n})= 0$ 
for any $t\in\C$ where $a=(a_1,\ldots,a_n)\in\N^n$ determines $\kappa$, 
which implies that $F$ is not nondegenerate. 
\end{proof}

From the definition of the nondegeneracy, 
it is important to understand
the geometrical properties of 
the singular varieties $V(F_{\kappa})$ 
in (\ref{eqn:3.2}).
The following is an interesting result
about $V(F_{\kappa})$. 

\begin{lemma}[\cite{BhS09}]
Let $\kappa$ be a bounded face of $\Gamma_+(F)$. 
If $F_{\kappa}$ is plurisubharmonic on $\C^2$, 
then $V(F_{\kappa})$ contains 
only finitely many complex curves. 
\end{lemma}



\section{A new invariant and adapted coordinates}

In this section, 
we introduce a new invariant and 
coordinates related to this invariant, 
which will be intrinsically important in this paper.

\subsection{The invariant $\rho_1(M,p)$}

From the definition of $\rho_1(M,p;(z))$, 
we can naturally define the following invariant:
\begin{equation}\label{eqn:4.1}
\rho_1(M,p)=\sup_{(z)}\{\rho_1(M,p;(z))\},
\end{equation}
where the supremum is taken over all 
holomorphic coordinates $(z)$ around $p$.
A given holomorphic coordinate $(z)$ at $p$ is called 
an {\it adapted coordinate} for $M$ at $p$ if 
the following equality holds:
\begin{equation}\label{eqn:4.2}
\rho_1(M,p;(z))=\rho_1(M,p).
\end{equation}
When $\rho_1(M,p)$ is finite, 
there exists an adapted coordinate for $M$ at $p$
since $\rho_1(M,p;(z))$ takes a positive integer.
But, when $\rho_1(M,p)=\infty$, 
there does not always exist an adapted coordinate.
Indeed, let us consider the case of the real hypersurface 
$M\subset\C^2$ 
in \cite{BlG77}, \cite{FoN} defined
by ${\rm Re}(z_2)+F(z_1,\bar{z}_1)=0$ 
where $F\in C^{\infty}_0(\C)$ 
admits the Taylor series $\sum_{j=2}^{\infty}
j! {\rm Re}(z_1^j)$. 
It can be easily seen that $\rho_1(M,0)=\infty$ but 
$\rho_1(M,0;(z))<\infty$ for any  
holomorphic coordinates $(z)$ at the origin. 
Note that regardless of the value of $\rho_1(M,p)$, 
canonical coordinates in Theorem~1.3 are always 
adapted coordinates.

Let us consider the relationship 
between the invariant (\ref{eqn:4.1}) and 
the {\it regular type} of $M$ at $p$ 
(see (\ref{eqn:1.2})). 
From the definition, 
it is easy to see the inequality:
$
\rho_1(M,p;(z))\leq \Delta_1^{{\rm reg}}(M,p)
$ 
for any coordinates $(z)$.
Thus, the following inequality always holds:
\begin{equation}\label{eqn:4.3}
\rho_1(M,p)\leq \Delta_1^{{\rm reg}}(M,p).
\end{equation}
Furthermore, we show that 
these two invariants always agree. 
\begin{proposition}
$\rho_1(M,p)=\Delta_1^{{\rm reg}}(M,p).$
\end{proposition}
\begin{proof}
From (\ref{eqn:4.3}),
it suffices to show that
$\rho_1(M,p)\geq\Delta_1^{{\rm reg}}(M,p)$.
We may assume that $p$ is the origin. 

First, let us consider the case when 
$\Delta_1^{{\rm reg}}(M,p)<\infty$.
Since $\ord(r\circ \gamma)$ is an integer, 
there exists a regular holomorphic mapping 
$\gamma\in\Gamma_{{\rm reg}}$ attaining  
$\Delta_1^{{\rm reg}}(M,p)
=\ord(r\circ \gamma)$. 
By the implicit function theorem,
without loss of generality,  
the map $\gamma$ may be expressed by 
$z_1=t$, 
$z_j=\varphi_j(t)$ $(j=2,\ldots,n)$,
where $\varphi_j\in{\mathcal O}_0(\C)$
with $\varphi_j(0)=0$. 
Let $(w)=(w_1,\ldots,w_n)$ be the holomorphic coordinate
defined by 
$w_1=z_1$, 
$w_j=z_j-\varphi_j(z_1)$ $(j=2,\ldots,n)$.
We denote 
$\varphi(w_1):=(\varphi_1(w_1),\ldots,\varphi_n(w_1))$.
Let $\tilde{r}$ be defined by 
$
\tilde{r}(w,\bar{w})=
r(w+\varphi(w_1),\overline{w+\varphi(w_1)}),
$
which is a local defining function for $M$ near $p$ 
on the coordinate $(w)$. 
Let $\gamma_*(t):=(t,0,\ldots,0)\in\Gamma$.  
Then we can see 
$\tilde{r}(\gamma_*(t),\overline{\gamma_*(t)})
=r(\gamma(t),\overline{\gamma(t)})$.
Since $
{\rm ord}(\tilde{r}\circ \gamma_*)
={\rm ord}(r\circ \gamma)
$, we have 
\begin{equation*}
\rho_1(M,p)\geq 
\rho_1(M,p;(w))=
{\rm ord}(\tilde{r}\circ \gamma_*)
={\rm ord}(r\circ \gamma)=\Delta_1^{{\rm reg}}(M,p).
\end{equation*}

Next, let us consider the case when
$\Delta_1^{{\rm reg}}(M,p)=\infty$. 
For any $N\in\N$, there exists 
a $\gamma_N\in{\Gamma}_{{\rm reg}}$ such that
${\rm ord}(r\circ\gamma_N)\geq N$. 
In a similar fashion to the above case 
of finite $\Delta_1^{{\rm reg}}(M,p)$,
we can show that,  
for any $N\in\N$, there exists 
a holomorphic coordinate $(w)$ such that
$\rho_1(M,p;(w))\geq N$. 
This means that $\rho_1(M,p)\geq N$ for
any $N\in\N$, which implies that 
$\rho_1(M,p)=\infty$.
\end{proof}

\subsection{Properties of adapted coordinates}

Let us consider a necessary condition for the adaptedness of  
coordinates. 
This condition will be useful for the investigation of 
the regular and singular types in Section~14.

Hereafter in this section, 
we assume that $F\in C^{\infty}_0(\C^n)$ is convenient 
(i.e. $\rho_j(F) < \infty$ for $j=1,\ldots,n$).
Without loss of generality, 
we may assume that $\rho_1(F)=\max\{\rho_j(F):j=1,\ldots,n\}$. 
Let ${\bf v}_1^*:=(\rho_1(F),0,\ldots,0)$, which 
is the vertex of 
${\mathcal N}_+(F)$ intersecting the $\xi_1$-axis.      
\begin{lemma}
Suppose that there exist a bounded face $\kappa$ of 
${\mathcal N}_+(F)$ and the complex curve
$\gamma\in\widetilde{\Gamma}^*_{\kappa}$ 
(see (\ref{eqn:3.1}))
with ${\rm ord}(\gamma)=1$
such that
$\kappa$ contains the vertex ${\bf v}_1^*$
and 
$F_{\kappa}\circ\gamma\equiv 0$.
Then there exists a local biholomorphic mapping
$\Psi:(\C^n,0)\to(\C^n,0)$  such that
$\rho_1(F\circ\Psi)>\rho_1(F)$.
\end{lemma}

\begin{proof}
Without loss of generality, we assume that
there exists a holomorphic mapping:
$\gamma(t)=(t,c_2 t^{a_2},\ldots, c_n t^{a_n})$ 
with $c_j\in\C^*$ and $a_j\in\N$ for $j=2,\ldots, n$
such that 
$\gamma\in\widetilde{\Gamma}^*_{\kappa}$ and 
$F_{\kappa}\circ\gamma \equiv 0$ on $\C$.
Lemma~7.1, below, implies 
\begin{equation}\label{eqn:4.4}
(F\circ\gamma)(t,\bar{t})=
(F_{\kappa}\circ\gamma)(t,\bar{t})+
R_{\kappa}(t,\bar{t})=R_{\kappa}(t,\bar{t}),
\end{equation}
where 
$R_{\kappa}\in\widetilde{\mathcal H}_{l+1}$ with $l=\rho_1(F)$. 
(Note that $\Phi(\gamma)=\Phi_*(\gamma)=\kappa$.) 

Now, let us define the local biholomorphic mapping:
$z=\Psi(w)$ as 
$z_1=w_1$, 
$z_j=w_j+c_j w_1^{a_j}$ $(j=2,\ldots,n)$.
Let $\gamma_*(t):=(t,0,\ldots,0)\in\Gamma$.
Since $\Psi\circ\gamma_*=\gamma$,
(\ref{eqn:4.4}) gives 
$$
(F\circ\Psi)(\gamma_*(t),\overline{\gamma_*(t)})=
(F\circ\gamma)(t,\bar{t})=
R_{\kappa}(t,\bar{t}).
$$
Therefore, we can see the following.
$$
\rho_1(F\circ\Psi)=
{\rm ord}(F\circ\Psi\circ\gamma_*)=
{\rm ord}(R_{\kappa})\geq l+1
>\rho_1(F).
$$
\end{proof}

\begin{definition}
Let $\kappa$ be a bounded face of a polyhedron $P\subset\R_{\geq}^n$.
We call $\kappa$ a {\it regular face} of $P$
if $\kappa$ is determined by a vector $a=(a_1,\ldots,a_n)\in\N^n$ 
satisfying that  $\min\{a_j:j=1,\ldots,n\}=1$.
\end{definition}

Let $r$ be a convenient local defining function for $M$ 
at $p$ on the coordinate $(z)$.
Let $m$ be the maximum integer such that
$\rho_j(r)=\rho_1(r)$ 
for $j=1,\ldots,m$. 
Let ${\mathcal V}$ be the set of 
vertices of ${\mathcal N}_+(r)$ defined by 
$
{\mathcal V}=
\{(0,\ldots,\stackrel{(j)}{\rho_1(r)},\ldots,0):j=1,\ldots,m\}$. 
 
\begin{lemma}
If $(z)$ is an adapted coordinate for $M$ at $p$, 
then $r_{\kappa}$ is nondegenerate
for any regular face $\kappa$ of 
${\mathcal N}_+(M,p;(z))$ intersecting 
${\mathcal V}$.
\end{lemma}

\begin{proof}
This can be easily shown by using Lemma~4.2. 
\end{proof}

As a corollary of Lemma~4.4, 
we can see the following. 
\begin{lemma}
Let $(z)$ be an adapted coordinate for $M$ at $p$. 
Suppose that the Newton diagram 
${\mathcal N}(M,p;(z))$ consists of only one facet
and that it is a regular face. 
Then $r_{\kappa}$ is nondegenerate
for each face $\kappa$ of 
${\mathcal N}_+(M,p;(z))$ intersecting 
${\mathcal V}$.
\end{lemma}

\begin{proof}
From Lemma~4.4, 
it suffices to check that
all proper faces of ${\mathcal N}(M,p;(z))$ 
are regular, which is easy. 
\end{proof}


\section{Proof of Theorem~1.1}

We are now in a position to prove Theorem~1.1.

From Proposition~4.1, 
it is easy to see that the inequalities:
\begin{equation}\label{eqn:5.1}
\Delta_1(M, p) \geq \Delta_1^{{\rm reg}}(M, p)= \rho_1(M,p) 
\geq \rho_1(M,p;(z))
\end{equation}
always hold for any holomorphic coordinates $(z)$ at $p$.
Recalling that  
an adapted coordinate $(z)$ at $p$ always exists
if $\rho_1(M,p)<\infty$,  
we can see that the equality in (i) implies (ii) from (\ref{eqn:5.1}).
The implication: (ii) $\Longrightarrow$ (i) 
is obvious from (\ref{eqn:5.1}). 

\begin{remark}
In Theorem~1.1, 
the assumption: $\Delta_1^{{\rm reg}}(M, p)<\infty$ 
can be weakened by the condition: 
an adapted coordinate exists for $M$ at $p$.
\end{remark}

\section{Order of vanishing}

Let $F\in C^{\infty}_0(\C^n)$ with $F(0,0)=0$ and let
$\Gamma$ be as in (\ref{eqn:3.1}) 

The order of vanishing of a nonflat function $F$ is defined by
\begin{equation}\label{eqn:6.1}
\begin{split}
{\rm ord}(F)&:=
\min\{
l\in\N: 
\exists(\alpha,\beta)\in \N_0^n\times\N_0^n \\
&\quad\quad\quad \mbox{with $|\alpha|+|\beta|=l$ 
such that 
$D^{\alpha}\bar{D}^{\beta}F(0,0)\neq 0$}\}.
\end{split}
\end{equation}
When $F$ is flat, we set ${\rm ord}(F):=\infty$.
For a vector-valued mapping 
$f=(f_1,\ldots,f_m)\in (C^{\infty}_0(\C))^m$, 
let 
$$\ord(f) : = \min\{{\rm ord}(f_j):j=1,\ldots, m\}.$$
We remark that 
${\rm ord}(\gamma)$ is a positive integer
for $\gamma\in\Gamma$. 
(Recall that $\Gamma$ does not contain the constant maps.)

For $l\in\Z_{\geq}$,
let ${\mathcal H}_l$ be the subset of $\C[t,\bar{t}]$ defined by 
\begin{eqnarray}\label{eqn:6.2}
{\mathcal H}_l=\left\{
\sum c_{jk} t^j \bar{t}^k:
j,k\in \N_0 \mbox{ with $j+k=l$ and $c_{jk}\in \C$}
\right\}, 
\end{eqnarray}
and let  $\widetilde{\mathcal H}_l$ be the  subset of 
$C^{\infty}_0(\C)$  
defined by 
\begin{equation}\label{eqn:6.3}
\widetilde{\mathcal H}_l=
\left\{
\sum  t^j \bar{t}^k a_{jk}(t,\bar{t}):
j,k\in\N_0 
\mbox{ with $j+k=l$ and $a_{jk}\in C^{\infty}_0(\C)$}
\right\}.
\end{equation}
The following properties of the two classes can be directly seen. 
Let $l, l_1,l_2\in\Z_{\geq}$.
\begin{itemize}
\item
${\mathcal H}_0=\C$ and $\widetilde{\mathcal H}_0=C^{\infty}_0(\C)$.
\item 
${\mathcal H}_l\subset\widetilde{\mathcal H}_l$.
If $l_1\leq l_2$, then
$\widetilde{\mathcal H}_{l_2}\subset\widetilde{\mathcal H}_{l_1}$.
\item 
${\mathcal H}_l$ and $\widetilde{\mathcal H}_l$ are vector spaces. 
Moreover, $\widetilde{\mathcal H}_l$ is an ideal of $C^{\infty}_0(\C)$.
\item If $f_1\in{\mathcal H}_{l_1}$ and $f_2\in{\mathcal H}_{l_2}$, 
then $f_1 f_2\in {\mathcal H}_{l_1+l_2}$.
\item 
If 
$f_1\in\widetilde{\mathcal H}_{l_1}$ and 
$f_2\in\widetilde{\mathcal H}_{l_2}$ with $l_1\leq l_2$, 
then $f_1+f_2 \in \widetilde{\mathcal H}_{l_1}$ and 
$f_1 f_2\in \widetilde{\mathcal H}_{l_1+l_2}$.
\item
If $f\in{\mathcal H}_l$,  
then the following two conditions are equivalent:
\begin{equation}\label{eqn:6.4}
{\rm (a)} \,\, \ord(f)=l, \quad 
{\rm (b)}\,\, f\not\equiv 0. 
\end{equation}
\item
If $f\in C^{\infty}_0(\C)$ with $f(0,0)=0$, 
then the following three conditions are equivalent:
\begin{equation}\label{eqn:6.5}
\begin{split}
{\rm (a)} \,\, f\in \widetilde{\mathcal H}_l, \quad 
{\rm (b)} \,\, \ord(f)\geq l, \quad 
{\rm (c)} \,\, f(t)=O(|t|^l). \quad
\end{split}
\end{equation}
\end{itemize}



\section{Asymptotics of $F\circ\gamma$}

Let $F\in C_0^{\infty}(\C^n)$ with $F(0,0)=0$ 
and let $\Gamma$ be as in (\ref{eqn:3.1}).

In this section, we compute the leading term of 
the asymptotic expansion of $F\circ\gamma$ 
for $\gamma\in\Gamma$
with respect to the classes ${\mathcal H}_l$ for
$l\in\N_0$.
This leading term can be determined 
by appropriate truncations of $F$ and $\gamma$. 

The computation is divided into the two cases.  
One of them is the case when 
$\gamma$ belongs to $\Gamma^*$
(see (\ref{eqn:3.1})). 
This case is considered as a generic one.  
The second case treats the curves
which are contained in some coordinate planes. 
(Note that the first case may be 
contained in the second one.) 
The second case seems very complicated but
this complexity is not essential.   
Essential analytic ideas of the computation 
can be explained 
by using the first generic case only.

\subsection{Asymptotics of $F\circ \gamma$ in the generic case}


Let 
${\mathcal F}_F^C$ denote the set of
bounded faces of ${\mathcal N}_+(F)$.
Exactly understanding the relationship between
$\Gamma^*$ and ${\mathcal F}_F^C$ 
is the most important to compute the 
asymptotics of $F\circ\gamma$ for $\gamma\in\Gamma^*$.  

Suppose that $F$ is not flat.  
We define the two important maps 
$l_*:\N^n\to\N$ and 
$\kappa_*:\N^n\to {\mathcal F}_F^C$
as follows:
\begin{equation}\label{eqn:7.1}
\begin{split}
&
l_*(a):=\min\left\{
\langle a,\xi \rangle : \xi \in {\mathcal N}_+(F) \right\}, \\
&
\kappa_*(a):=
\{\xi\in {\mathcal N}_+(F)  : \langle a,\xi \rangle=l_*(a)\}
(=H(a,l_*(a)) \cap {\mathcal N}_+(F) ).
\end{split}
\end{equation}
We remark that 
$\kappa_*(a)$ is bounded for $a\in\N^n$ 
from Lemma~2.1. 

Let us consider the sequence:
\begin{equation}\label{eqn:7.2}
\xymatrix{
    \Gamma^* \ar[r]^{\phi_*} \ar[d]_{\widetilde{}} & 
\N^{n}  \ar[r]^{\kappa_*} & 
{\mathcal F}_{F}^C 
\\
\widetilde{\Gamma}^*    \ar@{>}[ru]_{\phi_*}
  }
\end{equation}
%
where $\Gamma^*$, $\widetilde{\Gamma}^*$ are as in (\ref{eqn:3.1})
and the mappings $\tilde{\cdot}$, $\phi_*$ are defined as follows. 
\begin{itemize}
\item 
Since $\gamma_j\not\equiv 0$ 
for $j=1,\ldots, n$ and $\gamma_j(0)=0$, 
$\gamma_j(t)$ can be expressed 
by the convergence series expansion of $t$ as 
$\gamma_j(t) = c_j t^{a_j} + \cdots$
with some $a_j\in\N$ and $c_j\in \C^*$ for $j=1,\ldots, n$.
Then let $\tilde{\gamma}_j$ be the monomial defined by 
$\tilde{\gamma}_j(t)=c_j t^{a_j}$. 
The map $\widetilde{\cdot}$ 
from $\Gamma^*$ to $\widetilde{\Gamma}^*$ is defined as follows:
for $\gamma\in\Gamma^*$,
\begin{equation}\label{eqn:7.3}
\tilde{\gamma}(t)=
(\tilde{\gamma}_1(t), \ldots, \tilde{\gamma}_n(t))
=(c_1 t^{a_1},\ldots, c_n t^{a_n})\in\widetilde{\Gamma}^*.
\end{equation}
\item 
The map $\phi_*:{\Gamma}^* \to \N^n$ is defined by 
\begin{equation}\label{eqn:7.4}
\phi_*(\gamma)
=({\rm ord}(\gamma_1), \ldots, {\rm ord}(\gamma_n))\in\N^n.
\end{equation}
Note that $\phi_*(\gamma)=\phi_*(\tilde{\gamma})=(a_1,\ldots,a_n)=:a$
where $a\in\N^n$ is as in (\ref{eqn:7.3}).
%
\end{itemize}

When $F$ is not flat, the composition map 
$\Phi_*:\Gamma^*\to{\mathcal F}_F^C$ can be  
defined by 
\begin{equation}\label{eqn:7.5}
\Phi_*(\gamma)=(\kappa_*\circ\phi_*)(\gamma)
\end{equation}
from the sequence (\ref{eqn:7.2}).
For $\gamma\in\Gamma^*$, 
$\Phi_*(\gamma)$ is the bounded face defined by the 
hyperplane $H(\phi_*(\gamma),l_*(\phi_*(\gamma)))$. 

\begin{lemma}
Let $\gamma\in\Gamma^*$ and $l=l_*(\phi_*(\gamma))$. 
If $F$ is not flat, then
\begin{equation}\label{eqn:7.6}
F\circ \gamma \equiv F_{\Phi_*(\gamma)}\circ\tilde{\gamma}
\mbox{\,\,\, mod \,\,\, $\widetilde{\mathcal H}_{l+1}$},
\end{equation}
where $F_{\Phi_*(\gamma)}\circ\tilde{\gamma}$ belongs to 
${\mathcal H}_{l}$.
In particular, 
${\rm ord}(F\circ\gamma)\geq l_*(\phi_*(\gamma))$ holds.
\end{lemma}

\begin{proof}
Let $F$ admit the Taylor series (\ref{eqn:1.4}) 
at the origin.
In this proof, we set 
$$
\kappa=\Phi_*(\gamma),\quad
a=\phi_*(\gamma),\quad
l=l_*(\phi_*(\gamma))$$
for $\gamma\in\Gamma^*$.
Taylor's formula implies that 
\begin{equation}\label{eqn:7.7}
F(z,\bar{z})=F_{\kappa}(z,\bar{z})+R_{\kappa}(z,\bar{z}),
\end{equation}
where 
$R_{\kappa}\in C^{\infty}_0(\C^n)$ satisfies
${\mathcal N}_+(R_{\kappa})\subset H_+(a,l+1)$.
Substituting $z=\gamma(t)$ into (\ref{eqn:7.7}), we have
\begin{equation}\label{eqn:7.8}
(F\circ\gamma)(t,\bar{t})=
(F_{\kappa}\circ\gamma)(t,\bar{t})+
(R_{\kappa}\circ\gamma)(t,\bar{t}).
\end{equation}

First, let us consider the function 
\begin{equation}\label{eqn:7.9}
(F_{\kappa}\circ\gamma)(t,\bar{t})
=\sum_{\alpha+\beta\in \kappa} C_{\alpha\beta} 
\gamma(t)^{\alpha}\overline{\gamma(t)}^{\beta},
\end{equation}
where $C_{\alpha\beta}$ are as in (\ref{eqn:1.4}).
Note that each $\gamma_j$ can be expressed as 
$\gamma_j(t)=c_jt^{a_j}+r_j(t)$ where some 
$r_j\in {\mathcal O}_0(\C)$
with $r_j(t)=O(|t|^{a_j+1})$. 
Substituting these expressions into 
$\gamma(t)^{\alpha}$, we have
\begin{equation}\label{eqn:7.10}
\begin{split}
\gamma(t)^{\alpha}&=\prod_{j=1}^n \gamma_j(t)^{\alpha_j}
=\prod_{j=1}^n (c_jt^{a_j}+r_j(t))^{\alpha_j} \\
&=\prod_{j=1}^n (c_jt^{a_j})^{\alpha_j} +R_{\alpha}(t)
=c^{\alpha}t^{\langle a, \alpha \rangle} +R_{\alpha}(t),
\end{split}
\end{equation}
where $R_{\alpha}\in {\mathcal O}_0(\C)$. 
Moreover, it is easy to see that 
${\rm ord}(R_{\alpha})\geq \langle a, \alpha \rangle +1$, 
that is,  
$R_{\alpha}\in{\mathcal H}_{\langle a, \alpha \rangle+1}$.
Moreover, substituting (\ref{eqn:7.10}) into (\ref{eqn:7.9}), 
we have
\begin{equation*}
\begin{split}
(F_{\kappa}\circ\gamma)(t,\bar{t})
&=\sum_{\alpha+\beta\in \kappa} C_{\alpha\beta} 
c^{\alpha}\bar{c}^{\beta}
t^{\langle a, \alpha \rangle} 
\bar{t}^{\langle a, \beta \rangle}
+ \tilde{R}_{\kappa}(t,\bar{t}) \\
&=(F_{\kappa}\circ\tilde{\gamma})(t,\bar{t}) + \tilde{R}_{\kappa}(t,\bar{t}),
\end{split}
\end{equation*}
where $\tilde{R}_{\kappa}\in C^{\infty}_0(\C)$.
Moreover, 
it is easy to see that 
${\rm ord}(\tilde{R}_{\kappa})\geq \langle a, \alpha+\beta\rangle+1=l+1$,
that is, 
$\tilde{R}_{\kappa}\in\widetilde{\mathcal H}_{l+1}$.  
On the other hand,
since $\kappa$ is contained in the hyperplane $H(a,l)$, 
$F_{\kappa}\circ\tilde{\gamma}$ belongs to the class
${\mathcal H}_{l}$.    

Next, let us consider the function 
$R_{\kappa}\circ\gamma$.  
Taylor's formula implies that
\begin{equation}\label{eqn:7.11}
R_{\kappa}(z,\bar{z})=
\sum_{\alpha+\beta\in E} 
A_{\alpha\beta}(z,\bar{z}) 
z^{\alpha}\overline{z}^{\beta},
\end{equation}
where $E$ is a finite subset in $H_+(a,l+1)\cap \N_0^n$ 
and $A_{\alpha\beta}\in C_0^{\infty}(\C^n)$.
Note that each $\gamma_j$ can be expressed as 
$\gamma_j(t)=d_j(t)t^{a_j}$ where $d_j\in {\mathcal O}_0(\C)$
with $d_j(0)=c_j \,\,(\neq 0)$.
Substituting $\gamma_j(t)=d_j(t)t^{a_j}$ for $j=1,\ldots, n$ into 
$\gamma(t)^{\alpha}$, we have
\begin{equation}\label{eqn:7.12}
\begin{split}
\gamma(t)^{\alpha}&=\prod_{j=1}^n \gamma_j(t)^{\alpha_j}
=\prod_{j=1}^n (d_j(t)t^{a_j})^{\alpha_j} \\
&=\left(\prod_{j=1}^n d_j(t)^{\alpha_j}\right)
t^{\langle a, \alpha \rangle}
=d(t)^{\alpha}t^{\langle a, \alpha \rangle}.
\end{split}
\end{equation}

Moreover, substituting (\ref{eqn:7.12}) into (\ref{eqn:7.11}),
we have
\begin{equation}\label{eqn:7.13}
\begin{split}
(R_{\kappa}\circ\gamma)(t,\bar{t})
&=\sum_{\alpha+\beta\in E} B_{\alpha\beta}(t,\bar{t}) 
\gamma(t)^{\alpha}\overline{\gamma(t)}^{\beta} \\
&=
\sum_{\alpha+\beta\in E} 
B_{\alpha\beta}(t,\bar{t}) d(t)^{\alpha} \overline{d(t)}^{\beta}
t^{\langle a, \alpha \rangle} 
\bar{t}^{\langle a, \beta \rangle},
\end{split}
\end{equation}
where 
$B_{\alpha\beta}(t,\bar{t})=
A_{\alpha\beta}(\gamma(t),\overline{\gamma(t)})$. 
Since $\langle a, \alpha+\beta \rangle\geq l+1$ if 
$\alpha+\beta\in H_+(a,l+1)$, 
(\ref{eqn:6.5}) implies that if $\alpha+\beta\in E$, then
\begin{equation}\label{eqn:7.14}
\begin{split}
&{\rm ord}(B_{\alpha\beta}(t,\bar{t}) 
d(t)^{\alpha} \overline{d(t)}^{\beta}
t^{\langle a, \alpha \rangle} 
\bar{t}^{\langle a, \beta \rangle}) \\
=& 
{\rm ord}(B_{\alpha\beta}(t,\bar{t}) 
d(t)^{\alpha} \overline{d(t)}^{\beta})\cdot
{\rm ord}(t^{\langle a, \alpha \rangle})\cdot
{\rm ord}(\bar{t}^{\langle a, \beta \rangle}) \\
\geq& 
\langle a, \alpha \rangle + \langle a, \beta \rangle 
=\langle a, \alpha + \beta \rangle 
\geq l+1.
\end{split}
\end{equation}
Thus, from (\ref{eqn:7.13}), (\ref{eqn:7.14}), 
we see that $R_{\kappa}\circ\gamma$ 
belongs to the class $\widetilde{\mathcal H}_{l+1}$. 
\end{proof}

\subsection{The restriction process}

In order to treat the curves which are contained 
in some coordinate plane, 
we prepare many kinds of notation and symbols 
and must naturally generalize the maps  
constructed in the generic case.
Hereafter in this section, 
$I$ is a nonempty subset of $\{1,\ldots,n\}$. 

\subsubsection{Notation and symbols}

Let $K$ be one of the sets 
$\N, \N_0, \R, \R_{>}, \R_{\geq}, \C, \C^*$. 
We denote
\begin{equation}\label{eqn:7.15}
\widehat{K}:=K\cup\{\infty\}
\end{equation}
and 
\begin{equation}\label{eqn:7.16}
K^I:=\{(X_j)_{j\in I}:X_j\in K\}.
\end{equation}
For $X=(X_1,\ldots,X_n)\in K^n$, 
we denote $X_I:=(X_j)_{j\in I}\in K^{I}$.
\subsubsection{The maps $T_I^{(s)}$, $\iota_I^{(s)}$, $\pi_I$}
Let $s=0$ or $\infty$. 
The map $T_I^{(s)}:K^n\to\widehat{K}^n$ 
is defined by 
\begin{equation}\label{eqn:7.17}
(\hat{X}_1,\ldots,\hat{X}_n)=T_I^{(s)}(X_1,\ldots,X_n)\,\, 
\mbox{ with }\,\, \hat{X}_j:=
\begin{cases}
X_j& 
\quad \mbox{for $j\in I$}, \\
s&
\quad \mbox{otherwise}.
\end{cases}
\end{equation}
When $K=\R$ or $\C$, 
every coordinate plane in $K^n$ can be expressed by 
$T_I^{(0)}(K^n)$ for some $I\subset\{1,\ldots,n\}$, 
i.e.,
\begin{equation}\label{eqn:7.18}
T_I^{(0)}(K^n)=
\{X \in K^n:
X_j=0 \mbox{ if $j\not\in I$}\}.
\end{equation}
The map $\pi_I:K^n\to K^I$ is defined by
\begin{equation}\label{eqn:7.19}
\pi_I(X)=\pi_I(X_1,\ldots,X_n)
:=(X_j)_{j\in I}=X_I.
\end{equation}
Noticing 
$T_I^{(s)}(X)\in\widehat{K}^n$
depends only on $X_I$,  
we define an embedding map
$\iota_I^{(s)}:K^{I}\to\widehat{K}^n$ 
by 
\begin{equation}\label{eqn:7.20}
\iota_I^{(s)}(X_I):=T_I^{(s)}(X).
\end{equation}
It is easy to see that the following diagram commutes.
\begin{equation*}
\xymatrix{
  K^n
\ar[r]^{T_I^{(s)}} 
\ar[d]_{\pi_I} & 
\widehat{K}^n   
\\
K^I    \ar@{>}[ru]_{\iota_I^{(s)}}
  }
\end{equation*}
That is, $\iota_I^{(s)}\circ\pi_I=T_I^{(s)}$ holds.
When $I=\{1,\ldots,n\}$,
the maps
$T_I^{(s)}$, $\iota_I^{(s)}$, $\pi_I$ are 
the identity map 
(i.e.,
$T_I^{(s)}(X)=\iota_I^{(s)}(X)=\pi_I(X)=X$ 
for $X\in K^n$).

\subsubsection{The function $F_I$}


Let $z=(z_1,\ldots,z_n)\in\C^n$.
Let $F_I\in C^{\infty}_0(\C^I)$ be the germ of 
the function of $z_I$ 
$(=(z_j)_{j\in I}\in\C^{I}$) defined by 
\begin{equation}\label{eqn:7.21}
F_I(z_I):=(F\circ\iota_I^{(0)})(z_I).
\end{equation}
When $I=\{1,\ldots,n\}$, $F_I(z_I)=F(z)$.
Note that $F_I$ may be considered as 
the restriction of $F$ to 
the coordinate plane $T_I^{(0)}(\C^n)$ 
(see (\ref{eqn:7.18})).
The Newton polyhedron of $F_I$ is similarly defined, 
which is denoted by ${\mathcal N}_+(F_I)$. 
Note that 
${\mathcal N}_+(F_I)\subset\R_{\geq}^I$.


\subsubsection{A decomposition of $\widehat{\N}^n$}
Let us consider the case where $K=\N$ and $s=\infty$ in 
Section~7.2.2. 
It is easy to see that 
\begin{equation*}
\widehat{\N}^n \setminus\{\mbox{\boldmath $\infty$}\}
=\coprod_{I\subset\{1,\ldots,n\}} 
T_I^{(\infty)}(\N^n),
\end{equation*}
where ${\mbox{\boldmath $\infty$}}:=(\infty,\ldots,\infty)$
and 
the disjoint union is taken over 
all the nonempty subsets $I\subset\{1,\ldots,n\}$. 
Since the map $\iota_I^{(\infty)}$ is bijection, 
$T_I^{(\infty)}(\N^n)$ can be identified with $\N^I$. 
Therefore, 
$\widehat{\N}^n\setminus\{\mbox{\boldmath $\infty$}\}$ 
is decomposed 
into the following disjoint union:
\begin{equation}\label{eqn:7.22}
\widehat{\N}^n \setminus\{\mbox{\boldmath $\infty$}\}
=\coprod_{I\subset\{1,\ldots,n\}}\N^I.
\end{equation}
For $\hat{a}=(\hat{a}_1,\ldots,\hat{a}_n)
\in\widehat{\N}^n\setminus\{\mbox{\boldmath $\infty$}\}$, 
let $I(\hat{a})$ be the subset of $\{1,\ldots,n\}$ 
defined by 
\begin{equation}\label{eqn:7.23}
j\in I(\hat{a}) \Longleftrightarrow \hat{a}_j < \infty.
\end{equation}
When $\hat{a}_j<\infty$, we write $\hat{a}_j={a}_j$.
The map 
$\pi:\widehat{\N}^n\setminus\{\mbox{\boldmath $\infty$}\}
\to\coprod_{I\subset\{1,\ldots,n\}}\N^I$  
is defined by 
\begin{equation}\label{eqn:7.24}
\pi(\hat{a}):=(\hat{a}_j)_{j\in I}\,\,
(=(a_j)_{j\in I}=a_I)
  \quad
\mbox{ with $I=I(\hat{a})$ in (\ref{eqn:7.23}).}
\end{equation}
For example, 
it is easy to see the following two-dimensional case.
\begin{equation*}
\begin{split}
\widehat{\N}^2&=
\N^2\sqcup(\N\times\{\infty\})\sqcup
(\{\infty\}\times\N)\sqcup\{(\infty,\infty)\}\\
&=T_{\{1,2\}}^{(\infty)}(\N^2)\sqcup
T_{\{1\}}^{(\infty)}(\N^2)\sqcup
T_{\{2\}}^{(\infty)}(\N^2)\sqcup
\{\mbox{\boldmath $\infty$}\}\\
&=\N^{\{1,2\}}\sqcup\N^{\{1\}}\sqcup\N^{\{2\}}\sqcup
\{\mbox{\boldmath $\infty$}\}.
\end{split}
\end{equation*}

\subsubsection{The maps $l_I$, $\kappa_I$, $l$, $\kappa$}

We define the maps 
$l_I:\N^I\to \widehat{\N}$ and 
$\kappa_I:\N^I\to {\mathcal F}_{F_I}$
as follows.
When $F_I$ is not flat, set 
\begin{equation}\label{eqn:7.25}
\begin{split}
&l_I(a_I):=\min\left\{
\langle a_I,\xi_I \rangle_I : 
\xi_I \in {\mathcal N}_+(F_I) \right\}, \\
&
\kappa_I(a_I):=
\{\xi_I\in {\mathcal N}_+(F_I)  : 
\langle a_I,\xi_I \rangle_I=l_I(a_I)\}.
\end{split}
\end{equation}
When $F_I$ is flat, set $l_I(a_I):=\infty$ and $\kappa_I(a_I):=\emptyset$.

Now the maps: 
$l:\widehat{\N}^n\setminus\{\mbox{\boldmath $\infty$}\}\to \widehat{\N}$ 
and 
$\kappa:\widehat{\N}^n \setminus\{\mbox{\boldmath $\infty$}\}
\to {\mathcal F}_F$
are defined by
\begin{equation}\label{eqn:7.26}
\begin{split}
&l(\hat{a}):=(l_I\circ\pi)(\hat{a}), \\ 
&\kappa(\hat{a}):=(\iota_I^{(0)}\circ \kappa_I\circ\pi)(\hat{a}),
\end{split}
\end{equation}
where $I=I(\hat{a})$ in (\ref{eqn:7.23}),  
$\iota_I^{(0)}$ is as in (\ref{eqn:7.20}) and
$\pi$ is as in (\ref{eqn:7.24}).

\subsubsection{A decomposition of $\Gamma$}
Let 
\begin{equation}\label{eqn:7.27}
\Gamma^I:=\{(\gamma_j)_{j\in I}: 
\gamma_j\in{\mathcal O}_0(\C)
\mbox{ with } 
\gamma_j(0)=0 
\mbox{ and } \gamma_j\not\equiv 0
\mbox{ for $j\in I$} \}.\,\,
\end{equation}
In a similar fashion to the case of 
$\widehat{\N}^n\setminus\{\mbox{\boldmath $\infty$}\}$ in 
(\ref{eqn:7.22}), 
the set $\Gamma$ can be 
decomposed into the following disjoint union:
\begin{equation}\label{eqn:7.28}
\Gamma=\coprod_{I\subset\{1,\ldots,n\}} \Gamma^I.
\end{equation}
For $\gamma\in\Gamma$, 
let $I(\gamma)$ be a nonempty subset of $\{1,\ldots,n\}$ 
such that  
\begin{equation}\label{eqn:7.29}
j\in I(\gamma) \Longleftrightarrow \gamma_j\not\equiv 0.
\end{equation} 
(Recall that $\Gamma$ does not contain the map $\gamma\equiv 0$.)
The map $\tilde{\pi}:\Gamma\to
\coprod_{I\subset\{1,\ldots,n\}} \Gamma^I$
is defined by 
\begin{equation}\label{eqn:7.30}
\tilde{\pi}(\gamma)
=\tilde{\pi}(\gamma_1,\ldots,\gamma_n)
:=(\gamma_j)_{j\in I} 
\mbox{\, with \, $I=I(\gamma)$.}
\end{equation}


\subsection{Asymptotics of $F_I\circ \gamma_I$}

In this subsection, let 
$I\subset\{1,\ldots,n\}$ be arbitrarily given. 
For $z=(z_1,\ldots,z_n)\in\C^n$, 
let $F_I$ be as in (\ref{eqn:7.21}).
For $(a_I,l_I)\in\N_0^I\times\N_0$, 
define
\begin{equation}\label{eqn:7.31}
H_I(a_I,l_I)
:=\{\xi_I\in\R_{\geq}^I:
\langle a_I,\xi_I
\rangle_I=l_I
\}.
\end{equation}

Suppose that $F_I$ is not flat.  
Let us consider the sequence:
\begin{equation}\label{eqn:7.32}
\xymatrix{
    \Gamma^I \ar[r]^{\phi_I} \ar[d]_{\widetilde{}} & 
\N^{I}  \ar[r]^{\kappa_I} & 
{\mathcal F}_{F_I}^C 
\\
\widetilde{\Gamma}^I    \ar@{>}[ru]_{\phi_I}
  }
\end{equation}
where $\Gamma^I$ is as in (\ref{eqn:7.27}), 
${\mathcal F}_{F_I}^C$ 
is the set of bounded faces of ${\mathcal N}_+(F_I)$
and
$$
\widetilde{\Gamma}^I:=
\{(c_jt^{a_j})_{j\in I}:c_I\in(\C^*)^{I},a_I\in\N^{I}, 
t\in\C\},
$$
and the mappings 
$\tilde{\cdot}$, $\phi_I$
are defined in a similar fashion to the generic case
in (\ref{eqn:7.2}). 

In this setting, a similar argument to the generic case 
can be done in the space $\C^{I}$ or $\R_{\geq}^{I}$. 
When $F_I$ is not flat, 
the composition map 
$\Phi_I:\Gamma^I\to{\mathcal F}_{F_I}^C$ can be 
defined by 
\begin{equation}\label{eqn:7.33}
\Phi_I(\gamma_I):=(\kappa_I\circ\phi_I)(\gamma_I)
\end{equation}
from the sequence (\ref{eqn:7.32}).
For $\gamma_I\in\Gamma^I$, 
$\Phi_I(\gamma_I)$ is the bounded face defined by the 
hyperplane $H_I(\phi_I(\gamma_I),l_I(\phi_I(\gamma_I)))$,
where $H_I$ is as in (\ref{eqn:7.31}) and 
$l_I$ is as in (\ref{eqn:7.25}). 

\begin{lemma}
Let $I$ be an nonempty subset of $\{1,\ldots,n\}$. 
Let $\gamma_I\in\Gamma^I$ and $l=l_I(\phi_I(\gamma_I))$. 
If $F_I$ is not flat, then 
\begin{equation}\label{eqn:7.34}
F_I \circ \gamma_I \equiv (F_I)_{\Phi_I(\gamma_I)}\circ\tilde{\gamma}_I
\mbox{\,\, mod \,\,$\widetilde{\mathcal H}_{l+1}$},
\end{equation}
where $(F_I)_{\Phi_I(\gamma_I)}\circ\tilde{\gamma}_I$ belongs to 
${\mathcal H}_{l}$.
($(F_I)_{\Phi_I(\gamma_I)}$ means the
$\Phi_I(\gamma_I)$-part of $F_I$.) 
In particular, 
${\rm ord}(F_I\circ\gamma_I)\geq l_I(\phi_I(\gamma_I))$ holds.
\end{lemma}

\begin{proof}
The essential difference between Lemmas~7.1 and 7.2 is 
in the dimension. 
This lemma can be easily shown in a similar fashion to Lemma~7.1. 
\end{proof}


\subsection{Asymptotics of $F\circ \gamma$ in the general case}

Let us treat general complex curves in $\Gamma$. 

The three maps 
$\phi:\Gamma\to\widehat{\N}^n$, 
$\Phi:\Gamma\to{\mathcal F}_F^C$ and 
$\tilde{\cdot}:\Gamma \to \widetilde{\Gamma}$
are defined by
\begin{equation}\label{eqn:7.35}
\begin{split}
& \phi(\gamma):=
(\iota_I^{(\infty)}\circ\phi_I\circ\tilde{\pi})(\gamma), \\
& \Phi(\gamma):=(\iota_I^{(0)}\circ\Phi_I\circ\tilde{\pi})(\gamma), \\
& \tilde{\gamma}:=\iota_I^{(0)}(\widetilde{\tilde{\pi}(\gamma)}),
\end{split}
\end{equation}
for $\gamma\in\Gamma$, where $I=I(\gamma)$ (see (\ref{eqn:7.29})) 
and 
$\iota_I^{(s)}$ ($s=0$ or $\infty$) 
is as in (\ref{eqn:7.20}). 
When $\gamma\in\Gamma^*$, 
$\phi(\gamma)=\phi_*(\gamma)$,
$\Phi(\gamma)=\Phi_*(\gamma)$ 
and the last equation in (\ref{eqn:7.35}) is trivial. 
It is easy to see that
\begin{equation}\label{eqn:7.36}
I(\gamma)=I(\phi(\gamma)),\quad 
l(\phi(\gamma))=(l_I\circ\phi_I\circ\tilde{\pi})(\gamma)
\end{equation}
for $\gamma\in\Gamma$, where $I=I(\gamma)$ 
(see (\ref{eqn:7.23}), (\ref{eqn:7.26}), (\ref{eqn:7.29})).

\begin{theorem}
Let $\gamma\in\Gamma$ and set
$l=l(\phi(\gamma))$ and $I=I(\gamma)$.
If $F_I$ is not flat, then
\begin{equation}\label{eqn:7.37}
F\circ \gamma \equiv F_{\Phi(\gamma)}\circ\tilde{\gamma}
\mbox{\, mod \, $\widetilde{\mathcal H}_{l+1}$},
\end{equation}
where $F_{\Phi(\gamma)}\circ\tilde{\gamma}$ belongs to 
${\mathcal H}_{l}$.
In particular, 
${\rm ord}(F\circ\gamma)\geq l(\phi(\gamma))$ holds.
\end{theorem}

\begin{proof}
The definition of $I(\gamma)$ gives 
$F\circ\gamma=F_I\circ \gamma_I$.
Lemma~7.2 implies 
$$
F_I\circ \gamma_I\equiv (F_I)_{\Phi_I(\gamma_I)}\circ\tilde{\gamma}_I
\quad \mbox{mod $\widetilde{\mathcal H}_{l_I(\phi_I(\gamma_I))+1}$}.
$$
Noticing $(F_I)_{\Phi_I(\gamma_I)}\circ\tilde{\gamma}_I
=F_{\Phi(\gamma)}\circ\tilde{\gamma}$ and 
(\ref{eqn:7.36}), we can get the theorem.
\end{proof}

\begin{remark}
The above theorem shows that 
$F_{\Phi(\gamma)}$ and $\tilde{\gamma}$ are appropriate
truncations of $F$ and $\gamma$ in the computation 
of the order of contact of $F$ with $\gamma$. 
Indeed, if 
$F_{\Phi(\gamma)}\circ\tilde{\gamma}\not\equiv 0$, 
then $F$ and $\gamma$ in (\ref{eqn:1.12}) 
can be replaced by 
$F_{\Phi(\gamma)}$ and $\tilde{\gamma}$. 
The details will be discussed in the next section. 
\end{remark}


\subsection{Nondegeneracy condition on $F_I$}

Let $F_I$ be as in (\ref{eqn:7.21}) and 
let $\kappa_I$ be a bounded face of 
${\mathcal N}_+(F_I)$.
Let 
$\widetilde{\Gamma}_{\kappa_I}^*:=
\{(c_j t^{a_j})_{j\in I}: 
a_I \mbox{ determines } \kappa_I \}$,
where $c_I=(c_j)_{j\in I}\in(\C^*)^I$ and 
$a_I=(a_j)_{j\in I}\in \N^I$.
\begin{definition}
Let $\kappa_I$ be a bounded face of ${\mathcal N}_+(F_I)$.
The $\kappa_I$-part $(F_I)_{\kappa_I}$ of $F_I$ 
is said to be {\it nondegenerate} if 
\begin{equation*}\label{eqn:}
(F_I)_{\kappa_I}\circ \gamma_I \not\equiv 0 
\quad \mbox{ for any $\gamma_I\in \widetilde{\Gamma}^*_{\kappa_I}$}.
\end{equation*}
\end{definition}

Now, let us consider the case where 
a bounded face $\kappa$ of 
${\mathcal N}_+(F)$ is contained in some 
coordinate plane. 
Let  $I(\kappa)$ be the minimal subset $I$ of 
$\{1,\ldots, n\}$ such that 
$\kappa\subset T_I^{(0)}(\R^n)$ 
(see (\ref{eqn:7.18})). 
Hereafter we denote
$I:=I(\kappa)$ and 
$\kappa_I:=\pi_I(\kappa)\subset(\R_{\geq})^I$.
Note that $\kappa$ can be identified with $\kappa_I$ and 
that
$\kappa=\iota_I^{(0)}(\kappa_I)$.

\begin{lemma}
The following two conditions are equivalent.
\begin{enumerate}
\item $F_{\kappa}$ is nondegenerate; 
\item $(F_I)_{\kappa_I}$ is nondegenerate.
\end{enumerate}
\end{lemma}

\begin{proof}

((i) $\Longrightarrow$ (ii).) \quad 
Let $\gamma_I=(\gamma_j)_{j\in I}
\in\widetilde{\Gamma}^*_{\kappa_I}$ be arbitrarily given. 
Let $\check{\gamma}=(\check{\gamma}_1,\ldots,\check{\gamma}_n)
\in\Gamma$ be defined by 
$\check{\gamma}_j=\gamma_j$ 
if $j\in I$ and 
$\check{\gamma}_j(t)=t^m$ 
if $j\not\in I$ with $m\in \N$.
We can choose $m$ such that 
$\check{\gamma}$ belongs to $\widetilde{\Gamma}^*_{\kappa_I}$.
Note that $\gamma_I=\pi_I\circ\check{\gamma}$ and
$(F_I)_{\kappa_I}\circ\pi_I=F_{\kappa}$.
Since 
$(F_I)_{\kappa_I}\circ\gamma_I=
(F_I)_{\kappa_I}\circ\pi_I\circ\check{\gamma}=
F_{\kappa}\circ\check{\gamma}$, 
(i) implies $(F_I)_{\kappa_I}\circ\gamma_I\not\equiv 0$.

((ii) $\Longrightarrow$ (i).) \quad 
Let $\gamma=(\gamma_1\ldots,\gamma_n)\in\widetilde{\Gamma}^*_{\kappa}$
be arbitrarily given.  
Let $\gamma_I=(\gamma_j)_{j\in I}\in \Gamma^I$. 
Then $\gamma_I$ belongs to 
$\widetilde{\Gamma}^*_{\kappa_I}$.
Note that $\iota_I^{(0)}\circ\pi_I$ is 
the identity map. 
Since $F_{\kappa}\circ\gamma=
F_{\kappa}\circ\iota_I^{(0)}\circ\pi_I\circ\gamma
=(F_I)_{\kappa_I}\circ\gamma_I$, 
(ii) implies $F_{\kappa}\circ\gamma\not\equiv 0$.
\end{proof}

\section{Order of contact for smooth functions}

Let $F\in C^{\infty}_0(\C^n)$ with $F(0,0)=0$ and let
$\Gamma$ be as in (\ref{eqn:3.1}).


\subsection{Newton distance and $n$-tuple of numbers}

In this subsection, 
let 
$\hat{a}\in \widehat{\N}^n\setminus\{\mbox{\boldmath $\infty$}\}$
and set 
$I=I(\hat{a})$ (see (\ref{eqn:7.23})).
Note that $\hat{a}_I=
(\hat{a}_j)_{j\in I}=(a_j)_{j\in I}
=a_I\in\N^{I}$. 

The {\it Newton distance of $F$ in the direction} $\hat{a}$ is
defined by 
\begin{equation}\label{eqn:8.1}
d(F,\hat{a})
:=\frac{l(\hat{a})}{\min\{a_j:j=1,\ldots,n\}}
\left(=\frac{l_I(a_I)}{\min\{a_j:j\in I\}}\right),
\end{equation}
where $l$, $l_I$ are as in (\ref{eqn:7.26}), (\ref{eqn:7.25}).
We define the $n$-tuple of numbers:
$$
\rho(F,\hat{a}):=(\rho_1(F,\hat{a}),\ldots,\rho_n(F,\hat{a}))\in
\widehat{\N}_{0}^n
$$
as follows. 
\begin{enumerate}
\item 
When $F_I$ is not flat and $j\in I$, 
let $\rho_j(F,\hat{a})$ be the coordinate of 
the intersection of the hyperplane
$H_I(a_I,l_I(a_I))$ (see (\ref{eqn:7.31})) 
in $\R_{\geq}^I$ 
with the $\xi_j$-axis;
\item 
When $F_I$ is flat and  $j\in I$, 
we set $\rho_j(F,\hat{a})=\infty$;
\item 
If $j\not\in I$, then we set $\rho_j(F,\hat{a})=0$.
\end{enumerate}
When $F_I$ is not flat, 
we similarly define 
$
\rho(F_I,a_I)=
(\rho_j(F_I,a_I))_{j\in I}\in \N^I
$
as in (i).

The following lemma shows the relationship 
between $\rho(F,\hat{a})$ and $\rho(F)$ 
(see (\ref{eqn:2.9})).

\begin{lemma}
$
\rho_j(F,\hat{a})\leq \rho_j(F)$
for $\hat{a}\in \widehat{\N}^n\setminus\{\mbox{\boldmath $\infty$}\}$, 
$j=1,\ldots,n$.
\end{lemma}
\begin{proof}
From the geometrical relationship:
${\mathcal N}_+(F_I)\subset
H_+(a_I,l_I(a_I))$,  
the definitions of $\rho(F)$ and $\rho(F,\hat{a})$ easily imply 
the inequalities in the lemma. 
\end{proof}

By considering the geometrical meaning of $l(\hat{a})$,  
the quantity $d(F,\hat{a})$ stands for some kind of ``distance'' 
from the origin to ${\mathcal N}_+(F)$ in the direction $\hat{a}$.
The following lemma implies that 
this quantity can be expressed by using $\rho(F,\hat{a})$. 
Since $\rho(F,\hat{a})$ can be more directly computed, 
it will be convenient to see the value of $d(F,\hat{a})$. 

\begin{lemma}
$
d(F,\hat{a})=\max\{\rho_j(F,\hat{a}):j=1,\ldots,n\}
$ 
for $\hat{a}\in \widehat{\N}^n\setminus\{\mbox{\boldmath $\infty$}\}$.
\end{lemma}
\begin{proof}
Recalling that the hyperplane $H(\phi_I(\gamma_I),a_I)$
is defined by the equation: 
\begin{equation}\label{eqn:8.2}
\langle a_I,\xi_I\rangle_I 
=\sum_{j\in I} a_j\cdot \xi_j =l_I(a_I),
\end{equation}
we can see 
$\rho_{j}(F,\hat{a})=
l_I(a_I)\cdot a_j^{-1}$ for $j\in I$.
By using these equalities, 
\begin{equation*}
\begin{split}
&\max\{\rho_j(F,\hat{a}):j=1,\ldots,n\}=
\max\{\rho_j(F,\hat{a}):j\in I\}\\
&\quad=
\max\left\{
\frac{l_I(a_I)}{a_j}:j\in I
\right\}=
\frac{l_I(a_I)}{\min\{a_j:j\in I\}}
=d(F,\hat{a}).
\end{split}
\end{equation*}
\end{proof}


\subsection{Order of contact and the Newton distance}

Let us consider the relationship between 
the order of contact and the Newton distance. 
In this subsection, 
let $\gamma\in\Gamma$ be arbitrarily given and  
set $I=I(\gamma)$ (see (\ref{eqn:7.29})).

\begin{proposition}
{\rm (i)} \,${\rm ord}(F\circ \gamma)\geq l(\phi(\gamma))$, \quad
{\rm (ii)} \, 
$O(F,\gamma)
\geq d(F,\phi(\gamma)).$
\end{proposition}

\begin{proof}
When $F_I$ is flat, we see that 
${\rm ord}(F\circ\gamma)=
{\rm ord}(F_I\circ\gamma_I)=\infty$, 
which imply the estimates in (i), (ii). 
Therefore, 
it suffices to consider the case when $F_I$ is not flat.

The estimate in (i) has been stated in Theorem~7.3.
By noticing that 
(i) is equivalent to 
${\rm ord}(F_I\circ\gamma_I)\geq l_I(\phi_I(\gamma_I))$, 
(\ref{eqn:8.1}) gives 
\begin{eqnarray*}
O(F,\gamma)
=
O(F_I,\gamma_I)
\geq
\frac{l_I(\phi_I(\gamma_I))}{\min\{{\rm ord}(\gamma_j):j\in I\}} 
=d(F,\phi(\gamma)),
\end{eqnarray*}
which is the estimate in (ii). 

\end{proof}

The following theorem 
shows that the order of contact of $\gamma$ with $F$
exactly equals the distance of $F$ 
in the direction $\phi(\gamma)$ under the nondegeneracy 
assumption.

\begin{theorem}
If
$F_{\Phi(\gamma)}$  is nondegenerate, 
then we have
$$
{\rm (i)} \,\,\, {\rm ord}(F\circ \gamma)=l(\phi(\gamma)), \quad\,\,
{\rm (ii)} \,\,\, 
O(F,\gamma)=d(F,\phi(\gamma)).
$$
In particular, if $F$ is nondegenerate, 
then the equalities in (i), (ii) hold for any $\gamma\in\Gamma$. 
\end{theorem}


\begin{proof}
From the proof of Proposition~8.3,    
it suffices to show the following lemma. 
\end{proof}

\begin{lemma}
If $F_{\Phi(\gamma)}$ is nondegenerate, 
then $\ord(F_I\circ \gamma_I)=l_I(\phi_I(\gamma_I))$
with $I=I(\gamma)$. 
\end{lemma}

\begin{proof}
Note that  
$\Phi(\gamma)=\iota_I^{(0)}(\Phi_I(\gamma_I))$
from $I=I(\gamma)$.
First, let us consider the case when 
$F_I$ is not flat. 
Since the nondegeneracy condition of 
$F_{\Phi(\gamma)}$ 
implies that of $(F_I)_{\Phi_I(\gamma_I)}$ from Lemma~7.6,
we have 
${\rm ord}((F_I)_{\Phi_I(\gamma_I)}\circ\tilde{\gamma}_I)=
l_I(\phi_I(\gamma_I))$ from the equivalence (\ref{eqn:6.4}).
Therefore, it follows from the asymptotics in Lemma~7.2
that 
$$
{\rm ord}(F_I\circ \gamma_I)=
{\rm ord}((F_I)_{\Phi_I(\gamma_I)}\circ\tilde{\gamma}_I)= 
l_I(\phi_I(\gamma_I)).
$$

Next, let us consider the case when $F_I$ is flat.
Since
$F_I\circ \gamma_I$ is also flat, 
we have ${\rm ord}(F_I\circ \gamma_I)=\infty$. 
Recall that $l_I(\gamma_I)=\infty$,
then the equality in the lemma can be obtained. 
\end{proof}


\section{Proof of Theorem~1.3}

We are now in a position to prove Theorem~1.3.

Let $(z)$ be a canonical 
coordinate for $M$ at $p$.
Recall that $\rho_j(M,p;(z))=\rho_j(r)$ for $j=1,\ldots,n$
(see (\ref{eqn:2.9}), (\ref{eqn:2.11})).
From Theorem~1.1, it suffices to show that
$\Delta_1(M, p) \leq \rho_1(r)$.

Since the defining function $r$ for $M$ has the nondegeneracy condition
on $(z)$, 
Theorem~8.4 and Lemma~8.1 imply that  
the order of contact of $\gamma\in\Gamma$ 
with $M$ can be estimated as follows.
\begin{equation}\label{eqn:9.3}
\begin{split}
O(r,\gamma)&=d(r,\phi(\gamma))
=\max\{\rho_j(r,\phi(\gamma)):j=1,\ldots,n\}  \\
&\leq \max\{\rho_j(r):j=1,\ldots,n\}=\rho_1(r), 
\end{split}
\end{equation}
for any $\gamma\in\Gamma$.
This estimate gives
$\Delta_1(M, p) \leq \rho_1(r)$.

\section{On normalized coordinates}

It follows from Theorem~1.3 that 
when a hypersurface admits a canonical coordinate,
important properties of the regular and singular types
can be understood. 
On the other hand, it is another serious issue  
to determine the existence of canonical coordinates
for a given hypersurface and, 
if they exist, to actually construct these coordinates. 
The rest of this paper is to devote this issue. 
In this section, we prepare more convenient coordinates
for this investigation 
in special cases treated in Sections~11--15.  

Let $M$ be a real hypersurface in $\C^{n+1}$ and 
let $p$ lie on $M$. 


\subsection{Normalization of coordinates}

It is not always easy to check the  
nondegeneracy condition. 
This difficulty is sometimes caused by the existence of 
pure terms
in the principal part of a defining function. 
In order to avoid these situation, 
we prepare a convenient coordinate  
on which a local defining function is of more useful form.

\begin{lemma}
For any $N\in\N$, 
there exists a holomorphic coordinate 
$(z,w):=(z_1,\ldots,z_n,w)$ at $p$
on which a defining function $r$ for $M$ 
is expressed near the origin
as in the following form: 
\begin{equation}\label{eqn:10.1}
r(z,w,\bar{z},\bar{w})  
={\rm Re}(w)+ 
F(z,\bar{z})+R_1(z,\bar{z})\cdot{\rm Im}(w)+
R_2(z,w,\bar{z},\bar{w}),
\end{equation}
where
\begin{enumerate}
\item $F\in C^{\infty}_0(\C^n)$ satisfies that
\begin{enumerate}
\item $F(0,0)=0$ and $\nabla F(0,0)=0$;
\item If 
${\mathcal N}_+(F)\cap\{\xi:|\xi|\leq N\}\neq\emptyset$, 
then $j_N F$ (see (\ref{eqn:1.14})) contains no pure terms;
\item $\rho(F)$ satisfies
$\rho_j(F)=\rho_j(M,p;(z,w))\geq 2$ 
for $j=1,\ldots,n$
(see (\ref{eqn:2.9}), (\ref{eqn:2.11})); 
\end{enumerate}
\item $R_1\in C^{\infty}_0(\C^n)$ 
and $R_2\in C^{\infty}_0(\C^{n+1})$.  
Moreover, $R_2$ satisfies that
$|R_2(z,w,\bar{z},\bar{w})|\leq C |w|^2$
near $(z,w)=0$ where $C$ is a positive constant
independent of $(z,w)$.
\end{enumerate}
\end{lemma}

\begin{proof}
Taylor's formula easily implies all the above properties 
except (i-b).
A simple change of coordinates gives the property (i-b). 
\end{proof}

We call the above coordinate $(z,w)$
a {\it normalized coordinate} for $M$ at $p$ with $N\in\N$.
Of course, 
there are many such coordinates for a given $M$ and $N\in\N$.
Hereafter, normalized coordinates will be often used and 
the value of $N$ is, if possible, assumed to be so large that 
the principal part of $F$ contains no pure term
without any mentioning 
(see (\ref{eqn:1.15}) for the definition 
of the principal part).   
It can be easily seen that 
the above integer $N$ always exists 
when $\rho_1(M,p)<\infty$.  
On the other hand, 
in the case of the real hypersurface $M$ 
with $\rho_1(M,0)=\infty$ in \cite{BlG77}, \cite{FoN}
explained in Section~4.1, 
for every integer $N$ no matter how large, 
a pure term necessarily appears in the principal part of $F$
on normalized coordinates $(z,w)$ with $N$.

\begin{remark}
In the condition (ii) in Lemma~10.1, 
$|R_2|\leq C_2 |w|^2$ can be modified as 
$|R_2|\leq C_2 |{\rm Im}(w)|^2$. 
But, the estimate in (ii) is sufficient for our purpose
and is more convenient for the application 
(see Section~13). 
\end{remark}


\subsection{An equivalence condition for canonical coordinates}

After being normalized as in Lemma~10.1, 
it is relatively easy to check the canonical conditions. 

\begin{proposition}
Let $(z,w)$ be a normalized coordinate for $M$ at $p$ in Lemma~10.1. 
If the principal part $F_0$ of $F$ contains 
no pure terms, 
the following two conditions are equivalent.
\begin{enumerate}
\item A coordinate $(z,w)$ is canonical for $M$; 
\item $F$ is nondegenerate.
\end{enumerate}
\end{proposition}

This proposition easily follows from the following lemma.
\begin{lemma}
Let $r,F$ be as in Lemma~10.1. 
If the principal part $F_0$ of $F$ contains 
no pure terms, then
the following conditions are equivalent. 
\begin{enumerate}
\item 
$F$ is nondegenerate;
\item 
$r$ is nondegenerate.
\end{enumerate}
\end{lemma}
\begin{proof}
The implication: (ii) $\Longrightarrow$ (i) is obvious. 
In order to show the implication: (i) $\Longrightarrow$ (ii), 
it suffices to show that
$G_{\kappa}(z,w,\bar{z},\bar{w})=
2{\rm Re}(w)+F_{\kappa}(z,\bar{z})$ 
is always nondegenerate for every bounded face
$\kappa$ of ${\mathcal N}_+(F)$.
We assume that $G_{\kappa}$ is not nondegenerate.
Then, there exists a curve 
$\tilde{\gamma}=(\gamma,\gamma_{n+1})
:(\C,0)\to (\C^{n+1},0)$ contained in 
$(\C^*)^{n+1}\cup\{0\}$ 
such that $G_{\kappa}\circ\tilde{\gamma}\equiv 0$.
From this equation, 
we see $2{\rm Re}(\gamma_{n+1}(t))=
-F_{\kappa}(\gamma(t),\overline{\gamma(t)})$.
Now, $2{\rm Re}(\gamma_{n+1}(t))$ is harmonic, 
but $F_{\kappa}\circ\gamma$ $(\not\equiv 0)$ is not so, 
which is a contradiction. 
\end{proof}

\begin{remark}
In the above lemma, 
the assumption that the principal part $F_0$ of $F$ 
contains no pure terms 
is necessary. Indeed, 
$F(z,\bar{z})=|z_1-z_2|^2 + 2{\rm Re}(z_1 z_2)$ 
is nondegenerate, but
$r(z,w,\bar{z},\bar{w})=
2{\rm Re}(w)+F(z,\bar{z})$ is degenerate
(consider the curve: $\gamma(t)=(t,t,-t^2)$). 
\end{remark}

\subsection{On canonical normalized coordinates}

On the coordinates in Lemma~10.1, 
the regular and singular types can be determined by 
the information of $F$ only under nondegeneracy conditions.
More exactly,
these types are equal to the order of contact with the function $F$.
Let $\Gamma$ be as in (\ref{eqn:3.1}).

\begin{proposition}
Let $(z,w)$ be a normalized coordinate for $M$ at $p$ 
in Lemma~10.1. 
If the principal part $F_0$ of $F$ contains 
no pure terms and $F$ is nondegenerate,  
then
\begin{equation*}
\Delta_1(M, p) = 
\Delta_1^{{\rm reg}}(M, p) = 
\sup_{\gamma\in\Gamma} 
O(F,\gamma)
=\rho_1(F).
\end{equation*}
In particular, $M$ is of finite type at $p$ 
if and only if $F$ is convenient. 
\end{proposition}

\begin{proof}
Since $F$ is nondegenerate, 
Theorem~8.4 and Lemmas~8.1 and 8.2 imply
\begin{equation}\label{eqn:10.2}
\begin{split}
O(F,\gamma)
&=d(F,\phi(\gamma))
=\max\{\rho_j(F,\phi(\gamma)):j=1,\ldots,n\}\\
&\leq 
\max\{\rho_j(F):j=1,\ldots,n\}=\rho_1(F).
\end{split}
\end{equation}
Note that there exists a holomorphic mapping $\gamma$
attaining the equality in (\ref{eqn:10.2}).
Since 
$\rho_j(F)=\rho_j(r)$ for $j=1,\ldots,n$ and 
$\rho_{n+1}(r)=\rho_{n+1}(M,p;(z,w))=1$,
(\ref{eqn:10.2}) and Theorem~1.3 with Lemma~10.4 
imply
\begin{equation*}
\begin{split}
\sup_{\gamma\in\Gamma} O(F,\gamma)
=\rho_1(F)=\rho_1(r)
=\rho_1(M,p;(z,w))=\Delta_1(M, p). 
\end{split}
\end{equation*}
\end{proof}


\section{Pseudoconvexity}

Let $M$ be a real hypersurface in $\C^{n+1}$ 
and let $p$ lie on $M$. 
In the function theory of several complex variables, 
the case where $M$ is the boundary of 
a pseudoconvex domain is essentially important.
Let us investigate plurisubharmonicity and 
pseudoconvexity 
from the viewpoint of the notion of Newton polyhedra.
We say that a real hypersurface 
$M=\partial\Omega$ is pseudoconvex at $p$ if 
$\Omega$ is pseudoconvex at $p$. 

Let $F\in C^{\infty}_0(\C^n)$ with $F(0,0)=0$ and 
let $\Gamma$ be as in (\ref{eqn:3.1}). 


\subsection{Plurisubharmonicity}

The following lemma is valuable 
and has also been seen in \cite{FoS10}. 
Its proof is easy, so it is omitted.

\begin{lemma}
If $F\in C^{\infty}_0(\C^n)$ is plurisubharmonic, 
then so is the $\kappa$-part $F_{\kappa}$ of $F$
for every bounded face $\kappa$. 
\end{lemma}

We remark that 
the converse of the above lemma is not true.
For example, 
consider the function 
$F(z,\bar{z})=|z_1-z_2|^2-|z_2|^4$.


\subsection{Property {\bf PS}}

The notion of property {\bf PS} was introduced by D'Angelo 
\cite{Dan82}, \cite{Dan18},
which is a certain positivity condition, 
more general than plurisubharmonicity 
or pseudoconvexity. 
It is known that
this property is often sufficient for the investigation 
of order of contacts to see their important properties. 

\begin{definition}
We say that $F\in C^{\infty}_0(\C^n)$ 
satisfies {\it property} {\bf PS} at $0$ if,  
for any $\gamma\in\Gamma$ for which 
${\rm ord}(F\circ\gamma)<\infty$, 
\begin{enumerate}
\item ${\rm ord}(F\circ\gamma)$ is even, say $2m$;
\item $(d/dt)^m (d/d\bar{t})^m (F\circ\gamma)(0,0)> 0$.
\end{enumerate}
\end{definition}
From the definition, 
if $F$ is plurisubharmonic and 
the Taylor series of $F$ contains no pure terms, then
$F$ satisfies property {\bf PS} at $0$ 
(see \cite{Dan82}, \cite{Dan18}). 
Recall the {\it principal part} $F_0$ of $F$:
\begin{equation}\label{eqn:11.1}
F_{0}(z,\bar{z})=
\sum_{\a+\b \in {\mathcal N}(F)} 
C_{\a \b} z^\a \bar{z}^{\b},
\end{equation}
where ${\mathcal N}(F)$ is the Newton diagram of $F$  
(i.e., the union of 
all bounded faces of ${\mathcal N}_+(F)$).
\begin{proposition}
Suppose that the principal part $F_0$ of $F$ 
has no pure term. 
If $F$ satisfies property {\bf PS}, 
then 
the ${\bf v}$-part $F_{\bf v}$ of $F$ contains a term of the form 
$c|z_1|^{v_1}\cdots|z_n|^{v_n}$, 
with some $c>0$, 
for any vertex ${\bf v}=(v_1,\ldots,v_n)$ of ${\mathcal N}_+(F)$. 
\end{proposition}

\begin{remark}
The converse of the above proposition does not always hold. 
Indeed, consider the case when 
$F(z,\bar{z})=|z_1-z_2|^2+{\rm Re}(z_1^2\bar{z}_2^3)$
and $\gamma(t)=(t,t)$.
\end{remark}

\begin{proof}
Let $F$ admit the Taylor series (\ref{eqn:1.4}) at the origin. 

First, we consider the generic case:
the vertex  
${\bf v}=(v_1,\ldots,v_n)$ is away from any 
coordinate planes,  
i.e.,  ${\bf v}\in \N^n$ 
($\Leftrightarrow v_j>0$ for all $j$). 
Then there exists the set of 
valid pairs $\{(a^{(j)},l^{(j)}):j=1,\ldots,n\}$ 
for ${\mathcal N}_+(F)$
such that 
\begin{enumerate}
\item[(a)] The set of vectors: $\{a^{(j)}:j=1,\ldots,n\}$ 
is linearly independent; 
\item[(b)] $H(a^{(j)},l^{(j)})\cap {\mathcal N}_+(F)={\bf v}$.
\end{enumerate}

Let $\gamma^{(j)}(t):=(t^{a_1^{(j)}},\ldots,t^{a_n^{(j)}})$.
Since $\Phi(\gamma^{(j)})={\bf v}$ from the above (b),
$\gamma^{(j)}$ belongs to 
$\widetilde{\Gamma}_{\bf v}^*$ (see (3.1)).  
Lemma~7.1 implies 
$F(\gamma^{(j)}(t),\overline{\gamma^{(j)}(t)})=
F_{\bf v}
(\gamma^{(j)}(t),\overline{\gamma^{(j)}(t)})+o(t^{l^{(j)}})$.
Substituting 
$\gamma^{(j)}(t)=(t^{a_1^{(j)}},\ldots,t^{a_n^{(j)}})$
into the ${\bf v}$-part $F_{\bf v}$ of $F$,
we have
\begin{equation*}
F_{\bf v}(\gamma^{(j)}(t),\overline{\gamma^{(j)}(t)})=
\sum_{\alpha+\beta= {\bf v}}
C_{\alpha\beta}t^{\langle a^{(j)},\alpha\rangle}
\bar{t}^{\langle a^{(j)},\beta\rangle}.
\end{equation*}
Therefore, 
property {\bf PS} gives that,
for each $j=1,\ldots,n$, there exist 
$\alpha^{(j)},\beta^{(j)}\in\N_0^n$ 
such that
\begin{equation*}
\langle a^{(j)},\alpha^{(j)}\rangle=
\langle a^{(j)},\beta^{(j)}\rangle, \quad 
\alpha^{(j)}+\beta^{(j)}={\bf v}, \quad
C_{\alpha^{(j)}\beta^{(j)}}>0.
\end{equation*}
Since the linearly independence condition (a) implies 
$$
\bigcap_{j=1}^n 
\{\xi\in\N_0^n:
\langle a^{(j)},\xi\rangle=0
\}=\{0\},
$$
we can see that
$
\alpha^{(j)}=\beta^{(j)}$
and 
$\alpha^{(j)}+\beta^{(j)}={\bf v}$
for $j=1,\ldots,n$,
which imply 
\begin{equation*}
\alpha^{(j)}=\beta^{(j)}=\frac{1}{2}{\bf v}\,\, 
\mbox{ for $j=1,\ldots,n$.}
\end{equation*}

Next, we consider the general vertex 
${\bf v}=(v_1,\ldots,v_n)$
of ${\mathcal N}_+(F)$, 
i.e., ${\bf v}\in\N_0^n$. 
Let $I:=\{j:v_j\neq 0\}\subset\{1,\ldots,n\}$. 
Let $F_I$ be as in (\ref{eqn:7.21}). 
Since an essential difference between
$F_I$ and the above generic case of $F$
is only in the dimension,
the general case can be shown 
in a similar fashion to the generic case. 
\end{proof}

From Proposition~11.3, 
we can easily see the following.

\begin{corollary}
Suppose that $F$ satisfies property {\bf PS} at $0$. 
Then every component of the coordinate of the vertices of 
${\mathcal N}_+(F)$ is even. 
\end{corollary}

Let us recall the definition of property {\bf PS} 
in the case of real hypersurfaces in $\C^n$.

\begin{definition}
Let $M$ be a real hypersurface and let $p$ lie on $M$. 
Let us consider a normalized coordinate for $M$ at $p$ 
with large $N$ and a local defining function 
$r$ for $M$ near $p$ as in (\ref{eqn:10.1}) in Lemma~10.1. 
We say that $M$ satisfies property {\bf PS} at $p$
if 
there exists a $N_0\in\N$ such that
if $N\geq N_0$ then
$F$ in (\ref{eqn:10.1}) satisfies property {\bf PS} at $0$. 
\end{definition}

It is known that
every pseudoconvex hypersurface satisfies
property {\bf PS}. 
Refer to \cite{Dan82}, \cite{Dan18}
for more detailed explanation of property {\bf PS}.


\section{The semiregular ($h$-extendible) case}
The semiregular (or $h$-extendible) class was independently 
introduced in \cite{DiH94}, \cite{Yu94}.
This class contains many kinds of finite type pseudoconvex domains:
for example,  
strictly pseudoconvex domains in $\C^n$,
pseudoconvex domains in $\C^n$ whose boundaries have 
at most one zero eigenvalue in the Levi form
(in particular, the two-dimensional case), 
decoupled domains and convex domains in $\C^n$,
etc. 

This class can be variously characterized. 
The following are important.  
\begin{itemize}
\item The equality holds in the inequality between
the components of the
Catlin multitype and the D'Angelo variety type: 
$m_{n-q+1}=\Delta_q<\infty$ for $q=1,\ldots,n$ 
(the definition of semiregular in \cite{DiH94});
\item
They can be approximated by quasihomogeneous model 
(see Lemma~12.2 below);
\item 
There exists a bumping function satisfying quasihomogenous 
property
(the definition of $h$-extendible in \cite{Yu94}). 
\end{itemize}

\subsection{Quasihomogeneous model case}

In the quasihomogeneous model, 
we show that the nondegeneracy condition
exactly corresponds to the finite type condition. 

Let $P$ be a real-valued polynomial of $z\in\C^n$ having 
the property:
\begin{equation}\label{eqn:12.1}
P(r^{1/m}\bullet z,\overline{r^{1/m}\bullet z})=rP(z,\bar{z})
\quad \mbox{for all $r>0$, $z\in\C^n$, }
\end{equation}
where 
$m=(m_1,\ldots,m_n)\in\N^n$ with $m_j\geq 2$ and
\begin{equation}\label{eqn:12.2}
r^{1/m}\bullet z:=(r^{1/m_1}z_1,\ldots,r^{1/m_n}z_n).
\end{equation}
Let us consider the real hypersurface:
$$
M_P:=\{(z,w)\in\C^{n+1}:r_P(z,w,\bar{z},\bar{w}):=
{\rm Re}(w)+P(z,\bar{z})=0\}.
$$
Note that the property (\ref{eqn:12.1}) is equivalent to 
the condition that 
the Newton diagram 
${\mathcal N}(P)$ consists of only one facet  
defined by 
$\{\xi\in\R_{\geq}^n:
\sum_{j=1}^n \xi_j/m_j=1\}$. 

\begin{lemma}
As for the hypersurface $M_P$, 
the following two conditions are equivalent.
\begin{enumerate}
\item $P$ is convenient and nondegenerate;
\item The origin is of finite type.
\end{enumerate}
\end{lemma}

\begin{proof}
Easy.
\end{proof}


\subsection{Existence of canonical coordinates}

Let $\Omega$ be a domain with smooth boundary
and let $p$ lie on $\partial\Omega$.  
As mentioned in the beginning of this section, 
the following is one of the characterizations 
of semiregular class. 
\begin{lemma}[\cite{DiH94} Theorem 1.9, \cite{Yu94}]
The following two conditions are equivalent. 
\begin{enumerate}
\item[(1)]
$\partial\Omega$ is of semiregular type (h-extendible) at $p$;
\item[(2)]
There exists a normalized coordinate $(z,w)$ around $p$
on which a defining function for $\Omega$ near the origin
can be expressed as in (\ref{eqn:10.1}) in Lemma~10.1.
Here $R_1,R_2$ satisfy the condition (ii) in 
Lemma~10.1 and, moreover, 
the principal part $F_0$ (see (\ref{eqn:11.1})) of $F$ 
satisfies the following conditions: 
\begin{enumerate}
\item[(a)] $F_0$ contains no pure terms;
\item[(b)] $F_0$ is a plurisubharmonic polynomial; 
\item[(c)]
$F_0$ satisfies the equation:
\begin{equation}\label{eqn:12.3}
F_0(r^{1/m}\bullet z,\overline{r^{1/m}\bullet z})=rF_0(z,\bar{z})
\quad \mbox{for all $r>0$, $z\in\C^n$, }
\end{equation}
where $r^{1/m}\bullet z$ is as in (\ref{eqn:12.2}) and 
$m=(m_1,\ldots,m_n)\in\N^n$ 
with $m_1\geq \cdots\geq m_n\geq 2$;  
\item[(d)] As for the domain
$\Omega_0:=\{(z,w):{\rm Re}(w)+F_0(z,\bar{z})<0\}$,
the origin is of finite type.
\end{enumerate}
\end{enumerate}
\end{lemma}

From the above characterization, 
Theorem~1.3 directly implies the equality of regular and singular types.
Moreover, the existence of canonical coordinates can be seen. 

\begin{theorem}
If $\partial \Omega$ is of semiregular type at $p$, 
then there exists a canonical coordinate
for $\partial\Omega$ at $p$.
Indeed, the coordinate $(z,w)$ in Lemma~12.2 (2) itself
is a canonical coordinate for $\partial\Omega$ at $p$.  
\end{theorem}

\begin{proof}
We will show that 
the coordinate $(z,w)$ in Lemma~12.2 
is a canonical coordinate for $M$ at $p$. 
The condition (c) implies 
that $F(0,0)=0$ and $\nabla F(0,0)=0$. 
Moreover, Lemma~12.1 with the conditions (c) and (d) imply
that 
$F$ is nondegenerate
and that $\rho_j(F)=m_j$ for $j=1,\ldots,n$.
\end{proof}


\begin{remark}
The semiregular property can be characterized 
by using the shapes of Newton polyhedra on canonical coordinates. 
Let $\partial\Omega$ be of semiregular type at $p$. 
It follows from Lemma~12.2 that
there exists a coordinate around $p$ on which
a defining function $r$ for $\Omega$ is nondegenerate
and it has the Newton polyhedron
of the simple form: 
\begin{equation}\label{eqn:12.4}
\begin{split}
{\mathcal N}_+(r)&=
{\rm conv}(\{m_1{\bf e}_1,\ldots,m_n{\bf e}_n,{\bf e}_{n+1} \}) 
+\R_{\geq}^n \\
&=\left\{
\xi\in\R_+^{n+1}:
\sum_{j=1}^n \frac{\xi_j}{m_j} +\xi_{n+1}\geq 1
\right\}.
\end{split}
\end{equation}
In other words, 
the Newton diagram 
${\mathcal N}(r)$ consists of only one facet. 
\end{remark}

Before works \cite{DiH94}, \cite{Yu94}, 
McNeal \cite{Mcn92} and Boas and Straube \cite{BoS92}
showed that the singular type equals the line type
in the case of convex domains.
Here the {\it line type} is defined by 
the equation (\ref{eqn:1.2}) 
by replacing $\Gamma_{\rm reg}$ by the set of complex
lines through $p$.  

For a given real hypersurface,
it is not always simple matter 
to check the semiregularity condition 
and, moreover, to construct coordinates 
$(z,w)$ as in Lemma~12.2.  
J. Y. Yu \cite{Yu92} actually constructed a coordinate 
as in Lemma~12.2
in the case of convex domains.  


\section{The case of Reinhardt domains}

We first recall the definition of Reinhardt domains. 
For $\theta=(\theta_1,\ldots,\theta_n)\in\R^n$ 
and $z\in\C^n$, denote 
\begin{equation}\label{eqn:13.1}
e^{i \theta}\bullet z:=
(e^{i\theta_1}z_1,\ldots,e^{i\theta_n}z_n).
\end{equation}
A domain $\Omega\subset\C^n$ is Reinhardt 
if 
$e^{i\theta}\bullet z
\in\Omega$ whenever 
$z\in\Omega$ and $\theta\in\R^n$.

It has been shown by Fu, Isaev and Krantz \cite{FIK96} that 
if a hypersurface is the boundary of a pseudoconvex Reinhardt domain, 
then its regular and singular types agree. 
More strongly, 
we will show the existence of canonical coordinates 
in this case.


\subsection{Invariance under the rotation}

The function $F$ defined on $\C^n$ 
is said to be {\it invariant 
under the rotation} if $F$ satisfies 
\begin{equation}\label{eqn:13.2}
F(e^{i\theta}\bullet z,\overline{e^{i\theta}\bullet z})=
F(z,\bar{z})
\quad 
\mbox{ for all $\theta\in\R^n$.} 
\end{equation}

Let $M$ be a real hypersurface in $\C^{n+1}$ and 
let $p$ lie on $M$. 

\begin{theorem}
Let $(z,w)$ be a normalized coordinate for $M$ at $p$, 
on which 
$M$ is locally expressed by $r$ in 
(\ref{eqn:10.1}) in Lemma~10.1.
Suppose that $F$ is a nonflat plurisubharmonic function 
and that 
the principal part $F_0$ of $F$ is invariant 
under the rotation.
Then the coordinate $(z,w)$ is canonical for $M$ at $p$. 
\end{theorem}

\begin{proof}
Lemma~13.2, below, shows that 
$F_0$ satisfies the nondegeneracy condition. 
Thus, 
by noticing that the property (\ref{eqn:13.2}) 
implies that $F_0$ contains no pure terms,
Proposition~10.3 implies this theorem.
\end{proof}
\begin{lemma}
Let $P$ be a real-valued plurisubharmonic polynomial
$\not\equiv 0$.
If $P$ is invariant under the rotation,  
then $P$ is always positive on $(\C^*)^n$ and, 
in particular, $P$ is nondegenerate.
\end{lemma}

\begin{proof}
For $\zeta\in\C$, 
let 
$T_{k,c}(\zeta)
:=(c_1,\ldots,\stackrel{(k)}{\zeta},\ldots,c_n)$,
where $k=1,\ldots, n$ and $c_j\in\C$ are fixed for $j\neq k$.
Let $u:\C\to\R$ and $g:\R_{\geq}\to\R$  
be the functions defined by 
$$
u(\zeta,\overline{\zeta})
:=P(T_{k,c}(\zeta),\overline{T_{k,c}(\zeta)}),
\quad\,\,
g(r):=P(T_{k,c}(r),\overline{T_{k,c}(r)}).
$$
From the hypothesis of $P$, 
we see 
$$
g(re^{i\theta})=g(r) \mbox{ for any $\theta\in\R$ and }  
\triangle u\geq 0 \mbox{ on $\C$}.
$$ 
The above properties give an inequality:
\begin{equation}\label{eqn:11.2}
\frac{d}{dr}\left(r\frac{dg}{dr}\right)\geq 0 
\quad \mbox{ for $r\geq 0$.}
\end{equation}
By (\ref{eqn:11.2}),
$g$ must satisfy one of the following three conditions:
$$
\mbox{
$g\equiv 0$ on  $\R_{\geq}$;\quad
$g>0$ on $\R_>$ and $g(0)=0$;\quad
$g>0$ on $\R_{\geq}$.}
$$
By applying the above fact and by using
an easy inductive argument on the dimension, 
if there exists a point $z_0$ in $(\C^*)^n$ 
such that $P(z_0)=0$, then
$P\equiv 0$ on $\C^n$. 
This implies the assertion of the lemma.
\end{proof}

\begin{remark}
In Lemma~13.2,  
``plurisubharmonic'' cannot be replaced by 
``property {\bf PS}''. 
Consider the polynomial: 
$P(z,\bar{z})=(|z_1|^2-|z_2|^2)^2+|z_1|^8|z_2|^8$.
It is easy to see that $P$ satisfies property  
{\bf PS} at $0$ but $P$ is degenerate.
\end{remark}


\subsection{Construction of canonical coordinates}

Let us show the following theorem, 
which implies that the equality (\ref{eqn:1.3}) holds
in the pseudoconvex Reinhardt case.

\begin{theorem}
If $p$ is a boundary point of 
a pseudoconvex Reinhardt domain $\Omega$ with smooth boundary, 
then
there exists a canonical coordinate for $\partial\Omega$ at $p$. 
\end{theorem}

\begin{proof}
An important part of construction of  
a canonical coordinate has already been done in \cite{FIK96}. 
Refer for a substantial construction to \cite{FIK96}. 

Let $\Omega$ be in $\C^{n+1}$. 
Denote ${\mathcal Z}_j=
\{(z_1,\ldots,z_{n+1})\in\C^{n+1}:z_j=0\}$ for $j=1,\ldots,n+1$. 
Let ${\mathcal Z}=\cup_{j=1}^{n+1} {\mathcal Z}_j$. 
We define the holomorphic mapping as 
$$
L^*(z_1,\ldots,z_{n+1}):=
(\log z_1,\ldots,\log z_{n+1}),
$$
where the logarithm takes the principle branch, 
and $L^*$ is defined locally near every point of
$\C^n\setminus {\mathcal Z}$.

We divide the proof into two cases.

 ({\bf Generic case.})\quad
Let us first consider the case where 
$p\in\partial\Omega\setminus {\mathcal Z}$. 
As is well-known,
$L^*(\partial\Omega\setminus {\mathcal Z})$ 
is the boundary of a convex domain. 
This fact shows that this case is contained
in \cite{Mcn92}, \cite{BoS92} or 
in the semiregular case in Section~12.

({\bf The other case.}) \quad 
Let us consider the case where 
$p\in\partial\Omega\cap {\mathcal Z}$. 
It has been shown in \cite{FIK96} that 
there exists a coordinate $(z_1,\ldots,z_{n},\zeta)$ 
near $p$ such that $p=(0,\ldots,0,1)$,
on which 
$\partial\Omega$ can be expressed 
as 
$
\log|\zeta|+F(z,\overline{z})=0
$,
where $z=(z_1,\ldots,z_n)$ and 
$F$ satisfies all the properties 
in Theorem~13.1. 
By putting $w=\zeta-1$,
the coordinate $(z,w)$ is a desired
canonical coordinate around $(0,0)$. 
Indeed, $\partial\Omega$ can be locally expressed
on $(z,w)$ near the origin as 
${\rm Re}(w)+F(z,\bar{z})+R_2(w,\bar{w})=0$,
where $R_2$ is a smooth function of $w$ satisfying
$R_2(w)=(|w|^2)$.
Therefore, Theorem~13.1 implies that $(z,w)$ 
is a canonical coordinate
for $\partial\Omega$ around $p$.
\end{proof}

\begin{remark}
In the semiregular case, 
the respective Newton polyhedron essentially has 
a very simple form (see Remark~12.4).
In the Reinhardt case, 
the respective Newton polyhedra 
may take various kinds of shapes. 
For example, as for the real hypersurface in $\C^3$ 
defined by
$$
|z_1|^6+|z_2|^6+|z_1 z_2|^2+|z_3|^2-1=0,
$$
its Newton polyhedron at $(1,0,0)$ 
has two bounded facets. 
(This example is seen in \cite{FIK96}.)
\end{remark}


\section{The case where the regular type is $4$}

Let $M$ be a real hypersurface in $\C^{n+1}$ and
let $p$ lie on $M$. 
Recently, McNeal and Mernik \cite{McM18} and 
D'Angelo \cite{Dan18}
deeply investigated order of contact
in the case where $\Delta_1^{{\rm reg}}(M,p)=4$.
Indeed, it was shown in \cite{McM18}, \cite{Dan18} that
$\Delta_1(M, p) = \Delta_1^{{\rm reg}}(M, p)$ in (\ref{eqn:1.3}) 
always holds in this case under the assumption 
of pseudoconvexity or property {\bf PS} for $M$. 
By using the geometry of Newton polyhedra,  
we can see more precise geometrical structure 
of these hypersurfaces in the following theorem  
which, in particular, implies that 
the equality (\ref{eqn:1.3}) holds. 
\begin{theorem}
If a real hypersurface 
$M\subset\C^{n+1}$ satisfies property {\bf PS} at $p$ and 
$\Delta_1^{{\rm reg}}(M,p)=4$,
then 
there exists a canonical normalized coordinate $(z,w)$
for $M$ at $p$ on which
a local defining function for $M$ 
can be written as in (\ref{eqn:10.1}) in Lemma~10.1 with 
$F$ satisfying that 
$F$ is nondegenerate and 
the Newton diagram of $F$ can be expressed as 
\begin{equation*}
{\mathcal N}(F)={\rm conv}(\{4{\bf e}_1,\ldots,4{\bf e}_m,
2{\bf e}_{m+1},\ldots,2{\bf e}_n\}),
\end{equation*}
where
$m$ is an integer with $1\leq m\leq n$ 
(when $m=n$,  
${\mathcal N}(F)={\rm conv}(\{4{\bf e}_1,\ldots,4{\bf e}_n\}$).
Furthermore, if $M$ is pseudoconvex near $p$, 
then
its principal part $F_0$ 
of $F$ (see (\ref{eqn:11.1})) can be expressed as 
\begin{equation}\label{eqn:14.1}
F_0(z,\bar{z})=P(z',\overline{z'})+\sum_{j=m+1}^n |z_j|^2,
\end{equation}
$z'=(z_1,\ldots,z_m)$ and 
$P$ is a nondegenerate polynomial of 
$z'$ and $\overline{z'}$ 
whose Newton diagram is of the form:
\begin{equation}\label{eqn:14.2}
{\mathcal N}(P)={\rm conv}(\{4{\bf e}_1,\ldots,4{\bf e}_m\}).
\end{equation}
\end{theorem}

\begin{remark}
The above property (\ref{eqn:14.2}) 
is equivalent to the following conditions:
\begin{enumerate}
\item $P(r^{1/4}\bullet z',\overline{r^{1/4}\bullet z'})=
rP(z',\overline{z'})$ 
for all $r>0$ and $z'\in\C^m$;
\item $P$ is convenient, 
\end{enumerate}
where 
$
r^{1/4}\bullet z':=
(r^{1/4}z_1,\ldots,r^{1/4}z_m).
$
\end{remark}

\begin{proof}
Let $(z,w)$ be a normalized coordinate for $M$ with $N>4$, 
on which a local defining function can be expressed 
by using $F\in C^{\infty}_0(\C^n)$ as in 
(\ref{eqn:10.1}) in Lemma~10.1. 
From Corollary~11.5, 
property {\bf PS} implies that
every component of coordinate of 
vertices of ${\mathcal N}_+(F)$ is even.
Moreover,  the condition: $\Delta_1^{{\rm reg}}(M,p)=4$ 
implies that 
the Newton diagram ${\mathcal N}(F)$ is contained 
in the closed half space $\{\xi\in\R^n:|\xi|\leq 4\}$. 
Therefore, by simple combinatorial argument 
and the change of the variables if necessary,   
${\mathcal N}(F)$ can be essentially expressed as 
in a simple form:
\begin{equation}
\begin{split}
{\mathcal N}(F)&=
{\rm conv}(\{4{\bf e}_1,\ldots,4{\bf e}_l,
2{\bf e}_{l+1},\ldots,2{\bf e}_n\})\\
&=\left\{
\xi\in\R_{\geq}^{n}:
\sum_{j=1}^{l}\frac{\xi_j}{4}+
\sum_{j=l+1}^n\frac{\xi_j}{2}=1
\right\},
\end{split}
\end{equation}
where $l$ is some integer with $1\leq l\leq n$ 
(when $l=n$,  
${\mathcal N}(F)={\rm conv}(\{4{\bf e}_1,\ldots,4{\bf e}_n\}$).

We only consider the case when $1\leq l\leq n-1$. 
(The case when $l=n$ can be considered as a particular case.) 

Let $\kappa_1$ be the face of ${\mathcal N}_+(F)$
defined by 
$\kappa_1={\rm conv}(
\{2{\bf e}_{l+1},\ldots,2{\bf e}_n\})
$.
Then, the $\kappa_1$-part of $F$ is of the form:
$$
F_{\kappa_1}(z,\bar{z})=
\sum_{l+1\leq j,k\leq n}
C_{jk} z_j \bar{z}_k=:G(\tilde{z},\overline{\tilde{z}}), 
$$
where $C_{jk}\in \C$ satisfy 
$\overline{C_{jk}}=C_{kj}$ for 
$l+1\leq j,k\leq n$ and 
$\tilde{z}=(z_{l+1},\ldots,z_n)$.
There exists a unitary linear transform 
$U:\C^{n-l}\to \C^{n-l}$ such that
$$
G(U(\tilde{z}),\overline{U(\tilde{z})})=
\sum_{j=l+1}^n
\mu_j |z_j|^2  \quad 
\mbox{ with $\mu_j\in\R$}.
$$
From Proposition~11.3, 
property {\bf PS} implies $\mu_j\geq 0$ for $j=l+1.\ldots,n$.  
Without loss of generality, 
there exists $m\in\{l,\ldots, n-1\}$ such that 
$\mu_j>0$ if and only if $j\in\{m+1,\ldots,  n\}$. 
Now, we define $\hat{F}\in C^{\infty}_0(\C^n)$ as 
$$\hat{F}(z,\bar{z})=
F(z_1,\ldots,z_l,U(\tilde{z}),
\bar{z}_1,\ldots,\bar{z}_l,\overline{U(\tilde{z})}).$$
Since the Newton diagram 
${\mathcal N}(\hat{F})$ is also contained in 
the half space $\{\xi\in\R^n:|\xi|\leq 4\}$, 
we have
\begin{equation}
\begin{split}
{\mathcal N}(\hat{F})&=
{\rm conv}(\{4{\bf e}_1,\ldots,4{\bf e}_m,
2{\bf e}_{m+1},\ldots,2{\bf e}_n\}).
\end{split}
\end{equation} 

Let ${\mathcal V}$ be
the set  of vertices of ${\mathcal N}_+(\hat{F})$ defined 
by 
$
{\mathcal V}=\{{4\bf e}_j:j=1,\ldots,m\}.
$
Since the Newton diagram ${\mathcal N}(\hat{F})$ 
consists of only one facet which is regular, 
$\hat{F}_{\kappa}$ must be nondegenerate
for every face $\kappa$ intersecting ${\mathcal V}$
from Lemma~4.5.
On the other hand, 
every face not intersecting the set ${\mathcal V}$ is
of the form
$$
\kappa_I=
{\rm conv}(\{2{\bf e}_{j}:j\in I\})
\quad \, \mbox{for} \,\, I\subset\{m+1,\ldots,n\}.
$$
Since the ${\kappa}_I$-part of $\hat{F}$ is 
of the form 
$\hat{F}_{\kappa_I}(z,\bar{z})=
\sum_{j\in I} \mu_j |z_j|^2$
with $\mu_j>0$, 
which is nondegenerate.
We have checked the nondegeneracy condition 
for all the faces of ${\mathcal N}_+(\hat{F})$
and, as a result, 
$\hat{F}$ is nondegenerate.

Furthermore, adding the pseudoconvex assumption, 
we can see that  
the principal part of $\hat{F}$ is of the form:
\begin{equation}
\hat{F}_0(z,\bar{z})=
P(z',\overline{z'})+\sum_{j=m+1}^n \mu_j|z_j|^2,
\end{equation}
where $z'=(z_1,\ldots,z_m)$ and 
$P$ satisfies the properties in the theorem. 
After slightly changing of scaling, 
we obtain an appropriate coordinate
as in the theorem. 
\end{proof}

As a corollary of Theorem~14.1,
we easily obtain the following. 

\begin{corollary}
If $M$ is pseudoconvex at $p$ and 
$\Delta_1^{{\rm reg}}(M,p)=4$,
then $M$ is of semiregular type at $p$.
\end{corollary}

\begin{proof}
It follows form Theorem~14.1 that 
the Newton diagram ${\mathcal N}(F)$ consists 
of one facet. Moreover, from Lemma~11.1,  
the pseudoconvexity
implies the plurisubharmonicities of $P$ and $F_0$. 
These give the semiregularity of $p\in M$. 
\end{proof}

\begin{remark}
It has been shown in \cite{McM18} that
$\Delta_1^{{\rm reg}}(M,p)=2$ or $3$ 
respectively implies $\Delta_1(M,p)=2$ or 
$3$ without assuming properties {\bf PS}
(see also \cite{Koh79}).
Even if the pseudoconvexity is assumed, 
the equality (\ref{eqn:1.3}) can not be generalized 
to the case where $\Delta_1^{{\rm reg}}(M,0)=6$
(see (\ref{eqn:1.10})).
\end{remark}

\begin{remark}
D. Zaitsev (\cite{Zei17}, Theorem 2.1) has obtained a similar result 
to Theorem~14.1 
under the pseudoconvex assumption.
\end{remark}

\begin{remark}
McNeal and Mernik \cite{McM18} 
gave the following  interesting example: 
the real hypersurface $M$ in $\C^3$ defined by 
$r(z,w,\bar{z},\bar{w})={\rm Re}(w)+F(z,\bar{z})=0$ with
\begin{equation}
F(z,\bar{z})=|z_1|^2{\rm Re}(z_1^2-z_2^3)
+|z_2|^2{\rm Re}(z_2^2)-{\rm Re}(z_1^2 \bar{z}_2).
\end{equation}
This hypersurface does not satisfies property {\bf PS}
(in particular, pseudoconvexity) at the origin. 
They show that
$\Delta_1^{{\rm reg}}(M,0)=4$ but 
$\Delta_1(M,0)=\infty$.
Let us observe this example from our point of view. 
Theorem~1.3 implies that there is no
canonical coordinates for $M$ at the origin.  
The Newton diagram ${\mathcal N}(F)$ consists 
of the two facets: 
$$
\kappa_1:=
{\rm conv}(\{(4,0),(2,1)\}), \quad 
\kappa_2:=
{\rm conv}(\{(2,1),(0,4)\}).
$$
Notice the existence of the vertex at 
$(2,1)$, which is not compatible to property {\bf PS}
(see Corollary~11.5). 
The face $\kappa_2$ is not regular 
(the vector $(3,2)$ determines the face $\kappa_1$) 
and 
$F_{\kappa_2}(z,\bar{z})=
|z_2|^2{\rm Re}(z_2^2)-{\rm Re}(z_1^2 \bar{z}_2)$ 
is degenerate, 
which obstructs the construction of
canonical coordinates.   
\end{remark}

\begin{remark}
In order to obtain the clear form (\ref{eqn:14.1}) of the principal part $F_0$, 
the pseudoconvex assumption is necessary. 
Let us consider the real hypersurface $M$ in $\C^3$ defined by 
$r(z,w,\bar{z},\bar{w})={\rm Re}(w)+F(z,\bar{z})=0$ with
\begin{equation}
F(z,\bar{z})=|z_1|^4+2|z_1|^2{\rm Re}(z_2)+|z_2|^2.
\end{equation}
This hypersurface satisfies property {\bf PS} but is not 
pseudoconvex at the origin. 
\end{remark}


\section{The case of star-shaped domains}

Boas and Straube \cite{BoS92} 
not only give a simple geometrical proof 
of the McNeal's result \cite{Mcn92} 
concerning the convex domains
but also essentially generalize his result.

\begin{theorem}(Boas-Straube \cite{BoS92})
Let $M$ be defined by a defining function
$r$ of the form in (\ref{eqn:10.1}) in Lemma~10.1. 
Suppose that there is a real interval $[0,\delta]$ such that
for every fixed point $a=(a_1,\ldots,a_n)$ in the unit ball 
in $\C^n$, 
the function $r \mapsto F(r a, r \bar{a})$ is nondecreasing 
in $[0,\delta)$. 
Then the singular type is equal to the line type.
In particular, the equality (\ref{eqn:1.3}) holds 
for $M$ at the origin. 
\end{theorem}

Even if $M$ satisfies the hypothesis in the above theorem, 
$M$ does not always admit a canonical coordinate. 
For example, let us
consider the hypersurface defined by
$r_2$ in Remark~1.5. 
Indeed, it is easy to see that 
this hypersurface satisfies the hypothesis 
in Theorem~15.1,
but it does not admit any canonical coordinates.  

We remark that the hypothesis in the above theorem
is not necessary for the equality (\ref{eqn:1.3}).
Indeed, 
the hypothesis of
Theorem~15.1 does not contain 
the semiregular property. 
The Kohn-Nirenberg example in \cite{KoN73}, 
which is given by 
$r(z,w,\bar{z},\bar{w})=
{\rm Re}(w)+F_1(z,\bar{z})$ where
$F_1$ is as in (\ref{eqn:3.3}), 
does not satisfies the hypothesis 
under any coordinate changings, 
but it has been seen that 
this hypersurface admits a canonical coordinate 
and the equality (\ref{eqn:1.3}) holds.
By using results in \cite{KoN73}, \cite{Kol95},   
Theorem~1.3 easily provides many new examples 
satisfying (\ref{eqn:1.3}), which
are not contained in the cases discussed 
in Sections 12--15 (see Remark~1.5 (2)). 

\vspace{ .6 em}

{\bf Acknowledgements.} 
The author is grateful to Takeo Ohsawa, 
Natsuki Uehara and  Yoshiki Jikumaru 
for many valuable comments and discussion.
The author also greatly appreciates that the referee 
carefully and patiently read many reviced versions of 
this paper and gave many valuable comments 
which have greatly improved the readability of the paper.
Furthermore, Ninh Van Thu kindly informed the author of
important examples of infinite type real hypersurfaces in \cite{FoN}
with important suggestion,
which repaired some parts of discussion in Sections~4.1 and 10.1.   
This work was supported by 
JSPS KAKENHI Grant Numbers JP15K04932, JP15H02057.


\end{document}